\documentclass[letterpaper, 11pt,  reqno]{amsart}

\usepackage{amsmath,amssymb,amscd,amsthm,amsxtra, esint}
\usepackage[margin=1.2in,marginparwidth=1.5cm, marginparsep=0.5cm]{geometry}

\usepackage{mathtools}
\usepackage{color}
\usepackage[implicit=true]{hyperref}
\usepackage{accents}

\usepackage{cases}

\allowdisplaybreaks[2]

\usepackage{color}

\definecolor{gr}{rgb}   {0.,   0.69,   0.23 }
\definecolor{bl}{rgb}   {0.,   0.5,   1. }
\definecolor{mg}{rgb}   {0.85,  0.,    0.85}
\definecolor{yl}{rgb}   {0.8,  0.7,   0.}
\definecolor{or}{rgb}  {0.7,0.2,0.2}

\newtheorem{theorem}{Theorem} [section]

\newtheorem{lemma}[theorem]{Lemma}
\newtheorem{proposition}[theorem]{Proposition}
\newtheorem{remark}[theorem]{Remark}

\newtheorem{definition}[theorem]{Definition}

\newcommand{\I}{\mathcal{I}}

\newcommand{\noi}{\noindent}
\newcommand{\Z}{\mathbb{Z}}
\newcommand{\R}{\mathbb{R}}
\newcommand{\C}{\mathbb{C}}
\newcommand{\T}{\mathbb{T}}

\let\P= \undefined
\newcommand{\P}{\mathbf{P}}

\newcommand{\PP}{\mathcal{P}}

\newcommand{\E}{\mathbb{E}}

\newcommand{\F}{\mathcal{F}}

\newcommand{\al}{\alpha}
\newcommand{\be}{\beta}
\newcommand{\dl}{\delta}

\newcommand{\nb}{\nabla}

\newcommand{\Dl}{\Delta}
\newcommand{\eps}{\varepsilon}

\newcommand{\g}{\gamma}

\newcommand{\s}{\sigma}

\newcommand{\ft}{\widehat}

\newcommand{\cj}{\overline}

\newcommand{\dt}{\partial_t}

\renewcommand{\l}{\ell}
\renewcommand{\o}{\omega}
\renewcommand{\O}{\Omega}

\newcommand{\les}{\lesssim}

\newcommand{\jb}[1]
{\langle #1 \rangle}

\newcommand{\ind}{\mathbf 1}

\renewcommand{\S}{\mathcal{S}}

\newcommand{\D}{\mathcal{D}}

\newcommand{\gf}{\mathfrak{g}}

\newcommand{\nf}{\mathfrak{n}}
\newcommand{\mf}{\mathfrak{m}}

\newcommand{\N}{\mathbb{N}}
\newcommand{\NN}{\mathcal{N}}

\renewcommand{\H}{\mathcal{H}}

\newtheorem*{ackno}{Acknowledgements}

\newcommand{\too}{\longrightarrow}

\newcommand{\deff}{\stackrel{\textup{def}}{=}}

\newcommand{\wick}[1]{:\!{#1}\!:}

\newcommand{\Ff}{\mathfrak{F}}
\newcommand{\Gf}{\mathfrak{G}}

\DeclareMathOperator{\Sym}{\mathtt{Sym}}

\numberwithin{equation}{section}
\numberwithin{theorem}{section}

\makeatletter
\@namedef{subjclassname@2020}{%
  \textup{2020} Mathematics Subject Classification}
\makeatother

\begin{document}
\baselineskip = 14pt

\title[On  products of stochastic objects]
{A note on  products of stochastic objects}

\author[T.~Oh and Y.~Zine]
{Tadahiro Oh and Younes Zine}

\address{
Tadahiro Oh, School of Mathematics\\
The University of Edinburgh\\
and The Maxwell Institute for the Mathematical Sciences\\
James Clerk Maxwell Building\\
The King's Buildings\\
Peter Guthrie Tait Road\\
Edinburgh\\ 
EH9 3FD\\
 United Kingdom}

\email{hiro.oh@ed.ac.uk}

\address{
Younes Zine,  School of Mathematics\\
The University of Edinburgh\\
and The Maxwell Institute for the Mathematical Sciences\\
James Clerk Maxwell Building\\
The King's Buildings\\
Peter Guthrie Tait Road\\
Edinburgh\\ 
EH9 3FD\\
 United Kingdom}

\email{y.p.zine@sms.ed.ac.uk}



\begin{abstract}
In recent study of partial differential equations (PDEs) with random initial data 
and singular stochastic PDEs with random forcing, 
it is essential to study the regularity property of various 
stochastic objects.
These stochastic objects are often given as products
of simpler stochastic objects.
As pointed out in Hairer (2014), 
by using a multiple stochastic integral representation, 
one may use Jensen's inequality 
to reduce an estimate on the product
to those on simpler stochastic objects.
In this note, we present a simple argument
of the same estimate, based on Cauchy-Schwarz' inequality
(without any reference to 
multiple stochastic integrals).
We  present an example on
computing the regularity property of stochastic objects
in the study of the dispersion-generalized nonlinear wave equations,
and prove their local well-posedness
with rough random initial data.

\end{abstract}

%
\maketitle

\tableofcontents

\section{Introduction}
\label{SEC:1}

Over the last decade, we have seen a 
significant
progress in the theoretical study
of singular stochastic partial differential equations (PDEs), 
in particular
 in the parabolic setting
 \cite{Hairer, GIP}, 
 where we now have a satisfactory understanding of how
 to give a meaning to solutions to a certain class of classically ill-posed equations.
See \cite{CS} for a survey of the subject.
 Consider  an equation of the form:
\begin{align}
 \dt u = Lu + \NN(u) + \xi, 
\label{E1}
\end{align}
 
\noi
where $L$ denotes a linear operator in the spatial variable, 
$\NN(u)$ denotes a nonlinearity, 
and $\xi = \xi(x, t)$ denotes a random forcing.
 In constructing a solution to \eqref{E1}, 
we usually start out by 
constructing basic stochastic objects
such as the stochastic convolution:
\begin{align*}
Z (t)= \int_0^t e^{L (t - \tau)} \xi(d\tau)
\end{align*}

\noi
and  the second order term:
\begin{align*}
\I (\NN(Z))(t)  = \int_0^t e^{L (t - \tau)} \NN(Z)(\tau) d\tau
\end{align*}

\noi
(usually with a renormalization on the nonlinearity $\NN(Z)$).
Then, we 
expand the unknown~$u$ in terms of (explicitly known) basic stochastic objects
and study the equation for the remainder term \cite{BO96, DPD, Hairer, GIP}.
For example, with the second order expansion 
$u = Z + \I(\NN(Z)) + v$, 
we  study the equation for the remainder $v = u -Z - \I(\NN(Z))$:
\begin{align}
 \dt v = Lv + \NN(Z + \I(\NN(Z)) + v) -\NN(Z).
\label{E4}
\end{align}

\noi
In the case of a power-type nonlinearity $\NN(u) = u^k$, 
the nonlinearity in \eqref{E4} becomes 
\begin{align*}
\NN(Z + \I(\NN(Z)) + v) - \NN(Z)
= \sum_{\substack{k = k_1 + k_2 + k_3\\k_1 < k}} C_{k_1, k_2, k_3}
Z^{k_1}  \big(\I(Z^k)\big)^{k_2} v^{k_3}
\end{align*}

\noi
(once again with proper renormalizations), 
leading us to study  the regularity property
of the stochastic objects of the form:
\begin{align}
Z^{k_1}  \big(\I(Z^k)\big)^{k_2} .
\label{E5}
\end{align}
 

 In studying  the regularity property of stochastic objects,
 the following lemma (Lemma~\ref{LEM:reg}) provides a convenient criterion.
 Before stating the lemma, let us first introduce some notations.
See \cite{Kuo, Nua} for basic definitions.
Let $(H, B, \mu)$ be an abstract Wiener space.
Namely, $\mu$ is a Gaussian measure on a separable Banach space $B$ over reals
with $H \subset B$ as its Cameron-Martin space.
Given  a complete orthonormal system $\{e_j \}_{ j \in \N} \subset B^*$ of $H^* = H$, 
we  define a polynomial chaos of order
$k$ to be an element of the form $\prod_{j = 1}^\infty H_{k_j}(\jb{x, e_j})$, 
where $x \in B$, $k_j \ne 0$ for only finitely many $j$'s, $k= \sum_{j = 1}^\infty k_j$, 
$H_{k_j}$ is the Hermite polynomial of degree $k_j$, 
and $\jb{\cdot, \cdot} = \vphantom{|}_B \jb{\cdot, \cdot}_{B^*}$ denotes the $B$--$B^*$ duality pairing.
We then
define 
 the  homogeneous Wiener chaos $\H_k$ of order $k$
to be   the closure  of 
polynomial chaoses of order $k$ 
under $L^2(B, \mu)$.
Recall  the Ito-Wiener decomposition:
\[ L^2(B, \mu) = \bigoplus_{k = 0}^\infty \mathcal{H}_k,\]

\noi
where $\H_k$'s are pairwise orthogonal.
In the following, 
we
use  $\pi_{k}$ to  denote the orthogonal projection onto $\H_{k}$.
 For $k \in \N$, 
 we also  set 
\[ \H_{\leq k} = \bigoplus_{j = 0}^k \H_j.\]

\noi
For those who are not familiar with the subject, 
it is useful to  think of $\H_{\le k}$ as the collection
of  (possibly infinite) linear combinations
of polynomials of degree at most~$k$ in Gaussian random variables.
 We say that a stochastic process $X:\R_+ \to \mathcal{D}'(\T^d)$
is spatially homogeneous  if  $\{X(\cdot, t)\}_{t\in \R_+}$
and $\{X(x_0 +\cdot\,, t)\}_{t\in \R_+}$ have the same law for any $x_0 \in \T^d$.
Given $h \in \R$, we define the difference operator $\dl_h$ by setting
\begin{align}
\dl_h X(t) = X(t+h) - X(t).
\label{diff}
\end{align}

\begin{lemma}\label{LEM:reg}
Let $\{ X_N \}_{N \in \N}$ and $X$ be spatially homogeneous stochastic processes
$:\R_+ \to \mathcal{D}'(\T^d)$.
Suppose that there exists $k \in \N$ such that 
  $X_N(x, t)$ and $X(x, t)$ belong to $\H_{\leq k}$ for each $x \in \T^d$ and $t \in \R_+$.\footnote{Strictly speaking, since $X_N(t)$ and $X(t)$ are distributions, 
  we need to test them against test functions on $\T^d$ to say that they belong to $\H_{\le k}$.
  Namely, we require that $\jb{X_N(t), \varphi}, \jb{X(t), \varphi} \in \H_{\le k}$
  for any $t \in \R_+$ and $\varphi \in C^\infty(\T^d)$, where $\jb{\cdot, \cdot}$
  denotes the $\D'(\T^d)$-$C^\infty(\T^d)$ duality.
  Alternatively, we can require that the Fourier coefficients $\ft X_N(n, t)$
  and $\ft X(n, t)$ belong to $\H_{\le k}$ for any $t \in \R_+$ and $n \in \Z^d$. }

\smallskip
\noi\textup{(i)}
Let $t \in \R_+$.
If there exists $s_0 \in \R$ such that 
\begin{align}
\E\Big[ |\ft X(n, t)|^2\Big]\les \jb{n}^{ - d - 2s_0}
\label{reg1}
\end{align}

\noi
for any $n \in \Z^d$, then  
we have
$X(t) \in W^{s, \infty}(\T^d) \cap  \mathcal{C}^{s}(\T^d)$, $s < s_0$, 
almost surely.

\smallskip
\noi\textup{(ii)}
Suppose that $X_N$, $N \in \N$, satisfies \eqref{reg1}.
Furthermore, if there exists $\g > 0$ such that 
\begin{align*}
\E\Big[ |\ft X_N(n, t) - \ft X_M(n, t)|^2\Big]\les N^{-\g} \jb{n}^{ - d - 2s_0}
\end{align*}

\noi
for any $n \in \Z^d$ and $M \geq N \geq 1$, 
then 
$X_N(t)$ is a Cauchy sequence in $W^{s, \infty}(\T^d)  \cap  \mathcal{C}^{s}(\T^d)$, $s < s_0$, 
almost surely, 
thus converging to some limit 
 in  $W^{s, \infty}(\T^d)  \cap  \mathcal{C}^{s}(\T^d)$.

\smallskip
\noi\textup{(iii)}
Let $T > 0$ and suppose that \textup{(i)} holds on $[0, T]$.
If there exist $\s_1, \s_2 >0$ such that 
\begin{align}
 \E\Big[ |\dl_h \ft X(n, t)|^2\Big]
 \les \jb{n}^{ - d - 2s_0+ \s_1}
|h|^{\s_2}
\label{reg2}
\end{align}

\noi
for any  $n \in \Z^d$, $t \in [0, T]$, and $h \in [-1, 1]$
\textup{(}with  $h \geq - t$ such that $t + h \geq 0$\textup{)}, 
then we have 
$X \in C([0, T]; W^{s, \infty}(\T^d)\cap  \mathcal{C}^{s}(\T^d))$, 
$s < s_0 - \frac {\s_1}2$,  almost surely.

\smallskip
\noi\textup{(iv)}
Let $T > 0$ and suppose that \textup{(ii)} holds on $[0, T]$.
Furthermore, 
if there exists $\g > 0$ such that 
\begin{align}
 \E\Big[ |\dl_h \ft X_N(n, t) - \dl_h \ft X_M(n, t)|^2\Big]
 \les N^{-\g}\jb{n}^{ - d - 2s_0+ \s_1}
|h|^{\s_2}
\label{reg3}
\end{align}

\noi
for any  $n \in \Z^d$, $t \in [0, T]$,  $h \in [-1, 1]$, and $M\ge N \geq 1$, 
then 
$X_N$ is a Cauchy sequence in  $C([0, T]; W^{s, \infty}(\T^d) \cap  \mathcal{C}^{s}(\T^d))$, $s < s_0 - \frac{\s_1}{2}$,
almost surely, 
thus converging to some process in 
$C([0, T]; W^{s, \infty}(\T^d) \cap  \mathcal{C}^{s}(\T^d))$.

\end{lemma}

Here, $W^{s, p}(\T^d)$ denotes the $L^p$-based Sobolev space
and 
$\mathcal{C}^{s}(\T^d)$ denotes the H\"older-Besov
space $\mathcal{C}^{s}(\T^d) = B^{s}_{\infty, \infty}(\T^d)$.
In the case of the heat equation, 
we have $\s_1 = 2\s_2$, 
while we have $\s_1 = \s_2$
for the wave equation.
Lemma \ref{LEM:reg} follows
from a straightforward application of the Wiener chaos estimate
(Lemma~\ref{LEM:hyp}).
See \cite{MWX, OOTz} for the proof.

 \medskip

In the study of dispersive PDEs with random initial data and/or stochastic forcing, 
we have also witnessed a rapid progress over the recent years
\cite{BT2,  Poc, GKO, BOP3, OPTz, GKO2, OOk, Bring0, DNY1, STz1, OOTol, Bring, DNY2, STz2, OWZ, 
OOTol2}.
See~\cite{BOP4}
for a  (recent but already somewhat outdated) survey on the subject.
In the dispersive setting once again, 
we  need to study the regularity properties
of (explicitly given) stochastic objects, 
for example of the form \eqref{E5}, 
where $Z$ may be a random linear solution, 
a stochastic convolution, or their sum.
 Lemma \ref{LEM:reg}
 turns out to be a useful tool in this setting as well.\footnote{We point out that 
 Lemma \ref{LEM:reg} does not play an important role 
 in studying non-polynomial nonlinearities; see 
 \cite{ORSW1, ORSW2, ORW}.}
 We should, however, point out that 
 in the dispersive setting, 
a {\it multilinear dispersive smoothing} property often plays an important role
and computation of regularities of stochastic objects becomes
much more complicated than that in the parabolic case
 \cite{BO96, GKO2, DNY1, Bring, DNY2, OWZ}.
Furthermore, 
it is often necessary to 
start with stochastic objects which are not simply
 homogeneous random linear solutions \cite{Bring0, DNY1, DNY2, STz2}.

\medskip

As seen in \eqref{E5}, stochastic objects are often given as a (renormalized)
product of simpler stochastic objects.
For example, the term $Z^{k_1}  \big(\I(Z^k)\big)^{k_2}$ in \eqref{E5} is given as the product
of 
$Z^{k_1}$ and $  \big(\I(Z^k)\big)^{k_2} $, 
where the latter is further written as the $k_2$-fold product of 
$\I(Z^k)$.\footnote{In the singular setting, $Z$ is not a function
and thus we need to interpret $Z^{k_1}$ as a renormalized power, namely as a single object.}
In this case, 
we hope to study the regularity of 
 $Z^{k_1}  \big(\I(Z^k)\big)^{k_2}$ 
 via
Lemma \ref{LEM:reg}, using the (already studied) regularity
property (essentially the second moments of the Fourier coefficients) of $Z^{k_1}$  and $(\I(Z^k))^{k_2}$.
In computing the second moment of the Fourier coefficient 
of $Z^{k_1}  \big(\I(Z^k)\big)^{k_2}$, 
we need to take into account the interaction between
$Z^{k_1}$  and $\big(\I(Z^k)\big)^{k_2}$, 
which can be in general quite cumbersome,  when $k$, $k_1$,  and $k_2$ are large.
By using a multiple stochastic integral representation, however, 
we can invoke Jensen's inequality to ignore such
interaction between 
$Z^{k_1}$  and $\big(\I(Z^k)\big)^{k_2}$
in estimating the highest order contribution
(belonging
to $\H_{k_1 + kk_2}$)
from $Z^{k_1}\big(\I(Z^k)\big)^{k_2}$. 
This idea first appeared in \cite[Section 10]{Hairer}. 
See also \cite[Section~3]{MWX} and \cite[Appendix B]{OWZ}.
Our main goal in this note is to present this argument 
{\it without} going through a  multiple stochastic integral representation
so that it is more 
accessible to those in analysis and PDEs
who are not necessarily familiar with stochastic analysis
(in particular multiple stochastic integrals).
See Proposition \ref{PROP:main}\,(i).
Furthermore, by imposing more structures
on stochastic objects (see Assumption B below), 
we establish a  general product estimate.

Before we state out main result, let us introduce some notations. 
We first recall  the notion of a pairing from  
\cite[Definition 4.30]{Bring}.

\begin{definition} \rm
\label{DEF:pair}
Let $J\ge1$. We call a relation $\mathcal P \subset \{1,\dots,J\}^2$
a pairing if
\begin{enumerate}
\item[(i)] $\mathcal P$ is anti-reflexive,\footnote{In \cite{Bring}, it was referred 
to as reflexive.  We, however, change it to a more natural terminology.} i.e. $(j,j) \notin \mathcal P$ for all $1 \le j \le J$,

\smallskip

\item[(ii)] $\mathcal P$ is symmetric, i.e. $(i,j) \in \mathcal P$ if and only if $(j,i) \in \mathcal P$,

\smallskip

\item[(iii)] $\mathcal P$ is univalent, i.e. for each $1 \le i \le J$, $(i,j) \in \mathcal P$ for at most one $1 \le j \le J$.
\end{enumerate}

\smallskip

\noi
If $(i,j) \in \mathcal P$, the tuple $(i,j)$ is called a pair.
If $1 \le j \le J$ is contained in a pair, we say that $j$ is paired.
With a slight abuse of notation, we also write $j \in \mathcal P$ if $j$ is paired.
If $j$ is not paired, we also say that $j$ is unpaired and write $j \notin \mathcal P$.
Furthermore, 
given a partition $\mathcal A = \{A_\l\}_{\l = 1}^L$ 
 of $\{1, \cdots, J\}$, 
we say that $\mathcal P$ respects $\mathcal A$ if $i,j\in  A_\l$ 
for some $1 \le \l \le L$ implies that $(i,j)\notin \mathcal P$.
Namely, $\mathcal P$ does not pair elements of the same set $A_\l \in \mathcal{A}$.
We say that  $(n_1, \dots, n_J) \in (\Z^d)^J$ is  admissible 
with respect to $\PP$
 if $(i,j) \in \mathcal P$ implies that $n_i +  n_j = 0$.

\end{definition}

\noi
{\bf Setting:} Let $F_j$, $j = 1, \dots, J$, be  
stochastic processes$:\R_+\to \mathcal{D}'(\T^d)$
such that 
there exists  $k_j \in \N$
such that  $F_j(t) \in \H_{k_j}$ for any $t\in \R_+$.
Suppose that the Fourier coefficients
of $F_j(t)$ 
 are given by 
 \begin{align}
\ft F_j(n, t) = \sum_{n = n_1 + \cdots + n_{k_j}} 
\Ff_j(n_1, \dots, n_{k_j}, t)
\label{A1}
\end{align}

\noi
for any $n \in \Z^d$, where 
 $\Ff_j(n_1, \dots, n_{k_j}, t)$ 
satisfies
\begin{align}
 \E \bigg[ 
  \Ff_j(n_1, \dots, n_{k_j}, t_1) \cj{\Ff_j(m_1, \dots, m_{k_j}, t_2)}
 \bigg] 
&  = 0
\label{A2} 
\end{align}

\noi
for any $t_1, t_2 \in \R_+$, 
unless $\{n_1, \dots, n_{k_j}\} = \{m_1, \dots, m_{k_j}\}$, 
namely $m_\l = n_{\s(\l)}$ for some $\s \in S_{k_j}$, 
 where $S_{k_j}$ is the symmetric group on $k_j$ elements.

Next, we state two separate additional assumptions.
Assumption A is used in Part~(i) of Proposition~\ref{PROP:main}, 
where we  estimate the highest order contribution
(belonging to $\H_{K_J} = \H_{k_1 + \cdots k_J}$) of the product of $J$ homogeneous stochastic objects
without using Jensen's inequality.
Assumption B is used in Part (ii) of Proposition~\ref{PROP:main},
where we  estimate a general product
by imposing a further structure on $\Ff_j$.

\smallskip

\noi
{\bf Assumption A:}
Set $K_0 = 0 $ and  
\begin{align}
K_j = k_1 + \dots +k_j
\label{KK1}
\end{align}

\noi
for $j = 1, \dots, J$.
Recall that $\pi_k$  denotes the orthogonal projection onto $\H_k$.
We assume that, if the covariance
\begin{align}
  \E  \bigg[ 
\pi_{K_J }   \bigg(\prod_{j = 1}^J
\Ff_j(n_{K_{j-1} + 1}, \dots, n_{K_j}, t_1) \bigg)
\pi_{K_J }  \bigg(\prod_{j = 1}^J\cj{\Ff_j(m_{K_{j-1}+1}, \dots, m_{K_j}, t_2)}\bigg)
\bigg] 
\label{A3}
\end{align}

\noi
does not vanish, then
we have 
$\{n_1, \dots, n_{K_J}\} = \{m_1, \dots, m_{K_J}\}$
(counting multiplicity)
such that 
\begin{align}
\begin{split}
\eqref{A3}&  = 
\sum_{\s \in S_{K_J}}  \E  \bigg[  \pi_{K_J}  \bigg(\prod_{j = 1}^J
\Ff_j(n_{K_{j-1} + 1}, \dots, n_{K_j}, t_1 ) \bigg) \\
& \hphantom{lXXXX}
\times \pi_{K_J} \bigg(\prod_{j = 1}^J\cj{\Ff_j(n_{\s(K_{j-1}+1)}, \dots, n_{\s(K_j)}, t_2)}\bigg)
\bigg] 
\end{split}
\label{A4} 
\end{align}

\noi
 for any  $t_1, t_2 \in \R_+$.

Recalling that $F_j \in \H_{k_j}$, 
we see that the product $\prod_{j = 1}^J F_j$ belongs
to $\H_{\le K_J}$.
Assumption~A is introduced for estimating the highest order contribution
belonging to $\H_{K_J}$.
See Proposition~\ref{PROP:main}\,(i).
We also point out that Assumption A
is a very natural one to impose.
For example, it is easy to verify 
Assumption A if 
 $\Ff_j$ is a renormalized product of Gaussian random variables, namely, 
\[\Ff_j (n_1, \dots, n_{k_j}) = \pi_{k_j}\bigg(\prod_{i= 1}^{k_j} g_{n_i}\bigg),\]

\noi
where   $\{ g_n \}_{n \in \Z^d}$ is a family of 
independent standard   complex-valued  Gaussian random variables  
conditioned that  $g_{-n}=\cj{g_{n}}$,  $n \in \Z^d$.
Let us also point out that Assumption A
is satisfied under Assumption B
(in particular, \eqref{A5}, \eqref{A6}, and \eqref{A7})
that we introduce below.
See also Section \ref{SEC:3}
for a concrete example of the  dispersion-generalized nonlinear wave equations
with Gaussian  random initial data, satisfying Assumption A
(and also Assumption B).

\medskip

\noi
{\bf Assumption B:}
In order to fully describe the structure of a general product, 
let us give a more concrete representation of 
$\Ff_j(n_1, \dots, n_{k_j}, t)$
in case of a deterministic PDE with Gaussian random initial data.\footnote{A slight modification
 allows us to handle a stochastic PDE with an additive noise
but we do not pursue this issue here.  See \cite{Hairer, MWX, OWZ}.}
Given a  PDE, let $\I$ denote the Duhamel integral 
operator:
\begin{align*}
\I (f) (t)
: \! & = \int_0^t S(t - \tau) f(\tau) d\tau \\
& = \sum_{n \in \Z^d} \I_n(\ft f(n, \cdot\,)) (t) e^{in\cdot x} 
: = \sum_{n \in \Z^d}\bigg( \int_0^t 
S_n(t - \tau) \ft f(n, \tau) d\tau\bigg) e^{in\cdot x} , 
\end{align*}

\noi
where $S(t)$ is the linear propagator for the given PDE
and $S_n(t)$ is defined by 
$S_n(t)\ft f(n) = \F_x(S(t)f)(n)$.
Then, we assume that  $\Ff_j(n_1, \dots, n_{k_j}, t)$ is of the form:
\begin{align}
\Ff_j(n_1, \dots, n_{k_j}, t)
=  C_j(n_1, \dots, n_{k_j}, t)  \cdot
 D_j(\gf_{n_1}  , \dots, \gf_{n_{k_j}}) (t), 
\label{A5}
\end{align}

\noi
where  $C_j : (\Z^d)^{k_j} \times \R_+ \to \C$
is a deterministic function
and  
$D_j(\gf_{n_1}, \dots, \gf_{n_{k_j}})(t)$
denotes a  (Wick) renormalized product of Gaussian processes $\gf_{n_\l}(t)$, 
possibly with the Duhamel integral operator~$\I$, namely
\begin{align}
 D_j(\gf_{n_1} ,  \dots, \gf_{n_{k_j}})(t)
&   =  \pi_{k_j} \bigg(\prod_{i = 1}^{k_j} \gf_{n_i}(t)\bigg)
\label{A6}
\end{align}

\noi
or 
\begin{align}
\begin{split}
 D_j(\gf_{n_1},  \dots, \gf_{n_{k_j}})(t)
 & = \I_{n_1 + \cdots + n_{k_j}}  \bigg( \pi_{k_j}\Big(\prod_{i = 1}^{k_j} \gf_{n_i}\Big)\bigg)(t)\\
& = \int_0^t S_{n_1 + \cdots +n_{k_j}} (t - \tau)\, 
\pi_{k_j} \bigg( \prod_{i = 1}^{k_j} \gf_{n_i}(\tau)\bigg)
 d \tau.
 \end{split}
\label{A7}
\end{align}

\noi
Here,  
 $\pi_{k_j}$  denotes the orthogonal projection onto $\H_{k_j}$
 and 
 $\{\gf_n(t)\}_{n \in \Z^d}$ is a family 
of independent Gaussian processes conditioned that  $\gf_{-n}(t) = \cj{\gf_n(t)}$.\footnote{For simplicity
of the presentation, we focus on a real-valued PDE. A slight modification yields
a corresponding result for a complex-valued PDE such as the Schr\"odinger equation.}
In case of  a PDE with the linear part $\dt u - Lu$ (i.e.~with the first order derivative
in time) such as the KdV equation, 
$\gf_n(t)$ is simply a standard complex-valued Gaussian random variable (i.e.~independent of time $t$).
In case of a PDE with the linear part $\dt^2 u + Lu$
(i.e.~with the second order derivative in time), 
$\gf_n(t)$ is indeed time-dependent.
For example, for the  dispersion-generalized wave equation
with the linear part $\dt^2 u + (1-\Dl)^\al u$, 
we have 
\begin{align}
 \gf_n(t) = \cos (t \jb{n}^\al) g_n+ \sin (t \jb{n}^\al)h_n,
 \label{NLW0}
\end{align}

\noi
where $\jb{n} = \sqrt {1+|n|^2}$
and  the series   $\{ g_n ,  h_n \}_{n \in \Z^d}$ is a family of 
independent standard   complex-valued  Gaussian random variables  
conditioned that  $g_{-n}=\cj{g_{n}}$ and  $h_{-n}=\cj{h_{n}}$, $n \in \Z^d$.
See \eqref{NLW2} below.
Note that \eqref{A2} is clearly satisfied under \eqref{A5} and~\eqref{A6} (or \eqref{A7}).


Let $ N_j = \{K_{j-1} + 1, \dots, K_j\}$.
Given an integer $\l = 0, 1, \dots, \big[\frac{K_J}{2}\big]$, 
we define
 $\Pi_\l$ to be 
 the collection of pairings $\mathcal P$ on $\{1, \dots, K_J\}$
such that 

\begin{itemize}
\item[(i)] $\PP$ respects the partition 
\begin{align*}
\mathcal A= \{N_1, \dots, N_J\}
= \big\{\{1,\dots, K_1\}, \{K_1 + 1, \dots, K_2\}, 
\dots, \{K_{J-1} + 1, \dots,K_J\}\big\}, 
\end{align*}

\smallskip

\item[(ii)] $|\PP| =2 \l$ (when we view $\mathcal P$ as a subset of 
$\{1, \dots, n_{K_J}\}$).

\end{itemize}

\smallskip

\noi
With these notations, let us rewrite 
\begin{align}
\pi_{K_J - 2\l}   \bigg(\prod_{j = 1}^J
\Ff_j(n_{K_{j-1} + 1}, \dots, n_{K_j}, t) \bigg), 
\label{A8}
\end{align}

\noi
where $\pi_{K_J - 2\l}$  denotes the orthogonal projection onto $\H_{K_J-2\l}$.
Given  a partition  $A$ and $B$ of $\{1, \dots, J\}$, 
suppose that $\Ff_j$, $j \in A$, 
is given by~\eqref{A5} and  \eqref{A6}, 
 while $\Ff_j$, $j \in B$,  
is given by~\eqref{A5} and  \eqref{A7}. 
Then, we have
\begin{align}
\eqref{A8} 
= \sum_{\PP \in \Pi_\l}
  \ind_{\substack{(n_1, \dots, n_{K_J})\\\text{admissible}\\ \text{w.r.t.~} \PP}}
\cdot  \Gf_\PP(n_1, \dots, n_{K_J}, t).
\label{A9}
\end{align}

\noi
Here, with $\nf_j = n_{K_{j-1} + 1} + \cdots +  n_{K_j}$, 
the expression  $\Gf_\PP$ is given by 
\begin{align}
\Gf_\PP & (n_1  , \dots, n_{K_J}, t)
 =  \bigg(\prod_{j = 1}^J  C_j(n_{K_{j-1} + 1}, \dots, n_{K_j}, t) \bigg) \notag \\
& \quad \times 
\Bigg\{\int_{[0, t]^{|B|}}
\prod_{j \in B}
S_{\nf_j} (t - \tau_j)\cdot 
 Q_\PP(\{n_i, i \in \PP\},  \{\tau_j, j \in B\}, t) \notag  \\
&  \quad \times\pi_{K_J - 2\l}\Bigg( \bigg( 
\prod_{j \in B} \Big(\prod_{i \in N_j \setminus \PP}
\gf_{n_i}(\tau_j)\Big)\bigg)
\bigg(
 \prod_{i \in \bigcup_{j \in A} N_j \setminus \PP}
\gf_{n_i}(t)\bigg)\Bigg)
\prod_{j \in B}d \tau_j\Bigg\}
\label{A10}
\end{align}

\noi
for some deterministic function $Q_\PP =  Q_\PP(\{n_i, i \in \PP\},  \{\tau_j, j \in B\}, t) $.
The expression \eqref{A9}-\eqref{A10} follows
from \eqref{A8} with \eqref{A5}, \eqref{A6}, and \eqref{A7}
and contracting the frequencies $n_i$, $i \in \PP$.
For simplicity, let us consider the following example.
Let $g$ be a mean-zero real-valued Gaussian random variable with variance $\s$
i.e.~$g \sim \NN_\R(0,  \s)$, 
and  consider 
$\Ff ~= \prod_{j = 1}^J \Ff_j$
with $\Ff_j = g$, $j = 1, \dots, J$, 
and
the partition
$\mathcal{A} = \big\{\{1\}, \{2\}, \dots,  \{J\} \big\}$.
Then, as a corollary to 
Wick's theorem (see \cite[(1.19b)]{Simon}), 
 we have
\begin{align}
\begin{split}
\Ff = g^J 
& = \sum_{\l = 0}^{[J/2]}\frac{J!}{2^\l(J-2\l)! \l!} H_{J-2k}(g; \s)\cdot \s^\l\\
& = \sum_{\l = 0}^{[J/2]} \sum_{\PP\in \Pi_\l} \pi_{J - 2\l}(g^J) \cdot \s^{\l}, 
\label{A11}
\end{split}
\end{align}

\noi
where 
$H_k(x; \s) = \s^{\frac{k}{2}} H_k(\s^{-\frac 12}x)$; see \eqref{Wick1a} below.
Here, we used the fact that the number of $\l$ pairings is given by 
\begin{align*}
\sum_{\PP \in \Pi_\l} 1 =\frac{J!}{2^\l(J-2\l)! \l!}.
\end{align*}

\noi
Compare \eqref{A11} with \eqref{A8}, \eqref{A9}, and \eqref{A10}.
The power $\s^\l$ in \eqref{A11} represents the contribution
of the contracted terms, which appears in $Q_\PP$ in \eqref{A10}.

 We now state the main product estimates.

 \begin{proposition}\label{PROP:main}
Let $F_j$, $j = 1, \dots, J$,  be stochastic processes, satisfying 
\eqref{A1} and \eqref{A2}.

\smallskip

\noi
\textup{(i)} Suppose that, for each $j = 1, \dots, J$,
the term
 $\Ff_j(n_1, \dots, n_{k_j}, t)$ is symmetric in 
$n_1, \dots, n_{k_j}$ for any $t \in \R_+$
and satisfy Assumption A \textup{(}namely, \eqref{A4}\textup{)}.
Then, the following estimates hold\textup{:}
 \begin{align}
 \begin{split}
 \E  \bigg[ \Big| \pi_{K_J } \F_x\Big(\prod_{j = 1}^J F_j \Big)(n, t)\Big|^2\bigg] 
& \les 
\sum_{n = n_1 + \cdots + n_J} \,  
\prod_{j = 1}^J \E \Big[ |\ft F_j(n_j, t)|^2\Big]
\end{split}
\label{P1}
 \end{align}
 
 \noi
 and
 \begin{align}
 \begin{split}
  \E  &  \bigg[  \Big| \dl_h \pi_{K_J } \F_x\Big(\prod_{j = 1}^J F_j \Big)(n, t)\Big|^2\bigg] \\
& \les 
\sum_{\l = 1}^J
\sum_{n = n_1 + \cdots + n_J}
  \E \Big[ |\dl_h \ft F_\l (n_\l, t)|^2\Big]\\
& \hphantom{XXXX}
\times \bigg(  \prod_{j = 1}^{\l-1}
   \E \Big[ |\ft F_j(n_j, t+h)|^2\Big] \bigg)
\bigg(  \prod_{j = \l+1}^J
   \E \Big[ |\ft F_j(n_j, t)|^2\Big] \bigg)
\end{split}
\label{P2}
 \end{align}
 
 \noi
 for any  $n \in \Z^d$, $t \in \R_+$, and $h \in \R$
\textup{(}with  $h \geq - t$ such that $t + h \geq 0$\textup{)}, 
where $\dl_h$ is the difference operator defined in \eqref{diff}.

\smallskip

\noi
\textup{(ii)} 
Fix a partition   $A$ and $B$ of $\{1, \dots, J\}$.
Suppose that  $\Ff_j$, $j \in A$, 
is given by 
 \eqref{A5} and  \eqref{A6}, 
 while $\Ff_j$, $j \in B$,  
is given by 
 \eqref{A5} and  \eqref{A7}. 
Then, we have
\begin{align}
\begin{split}
\E & \bigg[ \Big|\F_x\Big(\prod_{j = 1}^J F_j \Big)(n, t)\Big|^2\bigg]
= \sum_{\l = 0}^{[K_J/2]}
\E \bigg[ \Big| \pi_{K_J - 2\l} \F_x\Big(\prod_{j = 1}^J F_j \Big)(n, t)\Big|^2\bigg]\\
& \les  \sum_{\l = 0}^{[K_J/2]} \sum_{\PP\in  \Pi_\l} 
 \sum_{\{n_i\}_{i\notin \PP}}
 \E \Bigg[ 
 \bigg|  
  \sum_{\{n_i\}_{i\in \PP}}
\ind_{n = n_1 + \cdots + n_{K_J}}\ind_{\substack{(n_1, \dots, n_{K_J})\\\textup{admissible}\\ \textup{w.r.t.~} \PP}}
\, \Gf_{\PP}(n_1, \dots, n_{K_J}, t)
\bigg|^2 \Bigg],
 \end{split}
\label{P3}
\end{align}

\noi
where $\Gf_\PP$ is as in \eqref{A9}.
In particular, 
from \eqref{P3} and \eqref{A10}, we have
\begin{align}
\E & \bigg[ \Big|\F_x\Big(\prod_{j = 1}^J F_j \Big)(n, t)\Big|^2\bigg]
= \sum_{\l = 0}^{[K_J/2]}
\E \bigg[ \Big| \pi_{K_J - 2\l} \F_x\Big(\prod_{j = 1}^J F_j \Big)(n)\Big|^2\bigg] \notag \\
&   \les   \sum_{\l = 0}^{[K_J/2]} \sum_{\PP\in  \Pi_\l} 
 \sum_{\{n_i\}_{i\notin \PP}}
 \E \Bigg[ 
 \bigg|  
  \sum_{\{n_i\}_{i\in \PP}}
\ind_{n = n_1 + \cdots + n_{K_J}}\ind_{\substack{(n_1, \dots, n_{K_J})\\\textup{admissible}\\ \textup{w.r.t.~} \PP}} \notag \\
& \quad 
\times \bigg(\prod_{j = 1}^J  C_j(n_{K_{j-1} + 1}, \dots, n_{K_j}, t) \bigg)\notag \\
& \quad \times 
\bigg\{
\int_{[0, t]^{|B|}}
\prod_{j \in B}
S_{\nf_j} (t - \tau_j)\cdot 
 Q_\PP(\{n_i, i \in \PP\},  \{\tau_j, j \in B\}, t) \notag \\
&  \quad \times\pi_{K_J - 2\l}\bigg( \bigg( 
\prod_{j \in B} \Big(\prod_{i \in N_j \setminus \PP}
\gf_{n_i}(\tau_j)\Big)\bigg)
\bigg(
 \prod_{i \in \bigcup_{j \in A} N_j \setminus \PP}
\gf_{n_i}(t)\bigg)\bigg)
\prod_{j \in B}d \tau_j
\bigg\}\bigg|^2 \Bigg].
\label{P4}
\end{align}

 \noi
 As in Part \textup{(i)}, 
a similar bound also holds
 for the time difference
$ \dl_h  \F_x\big(\prod_{j = 1}^J F_j \big)(n, t)$.
 
 \end{proposition}

For simplicity, we stated Proposition \ref{PROP:main}
for a usual product but 
a slight modification covers the case of 
 paraproducts and  resonant products.
With a multiple stochastic integral representation, 
Proposition \ref{PROP:main}\,(i) follows from Jensen's inequality;
see Appendix B in \cite{OWZ}.
In Section \ref{SEC:2}, we present a proof without a multiple stochastic integral representation.

Proposition \ref{PROP:main}\,(ii) is in the spirit of 
Lemma 4.1 in \cite{DNY1}.
In the dispersive setting, 
we typically apply Proposition~\ref{PROP:main}
and then apply a counting argument
(to the right-hand side of~\eqref{P4})  to exhibit multilinear dispersive smoothing.
In the non-dispersive setting, 
for example in the case of the nonlinear Schr\"odinger equation
with the Grushin Laplacian studied in~\cite{GL}, 
there is no such multilinear dispersive smoothing
and thus we can simply apply Proposition~\ref{PROP:main}
to estimate products of stochastic objects.
In treating a nonlinearity of a very high degree, in particular
in a weakly dispersive setting, 
Proposition~\ref{PROP:main}
(without a counting argument)
provides a sufficiently good control 
on products of stochastic objects.

\begin{remark}\rm
(i) Proposition \ref{PROP:main}\,(i) with Lemma \ref{LEM:reg}
allows us to immediately determine the regularity of the Wick powers
$:\! Z^k\!: \, = \pi_k(Z^k)$ of a basic stochastic object (such as a random linear solution
or a stochastic convolution).
See the proof of Proposition \ref{PROP:sto1}\,(i)
for such an application of 
Proposition \ref{PROP:main}\,(i).

\smallskip

\noi
(ii)
By using a symmetrization, 
we can always satisfy  the symmetry assumption in Proposition~\ref{PROP:main}\,(i).
More precisely, given $F_j$ in \eqref{A1}, 
we have
 \begin{align*}
\ft F_j(n, t) & = \sum_{n = n_1 + \cdots n_{k_j}} \Ff_j(n_1, \dots, n_{k_j}, t)
= \frac1{k_j!}\sum_{\s \in S_{k_j}} 
\sum_{n = n_1 + \cdots n_{k_j}} \Ff_j(n_{\s(1)}, \dots, n_{\s(k_j)}, t)\\
& = 
\sum_{n = n_1 + \cdots n_{k_j}} \Sym(\Ff_j)(n_{1}, \dots, n_{k_j}, t), 
\end{align*}

\noi
where $\Sym(\Ff_j)$ is the symmetrization of $\Ff_j$ defined by 
 \begin{align*}
 \Sym(\Ff_j)(n_{1}, \dots, n_{k}, t)
= \frac1{k_j!}\sum_{\s \in S_{k_j}}  \Ff_j(n_{\s(1)}, \dots, n_{\s(k_j)}, t).
\end{align*}

\noi
Hence, without loss of generality, 
we may assume 
that  $\Ff_j$ is symmetric.
%
%
Under the symmetry assumption, 
the condition \eqref{A2}  reduces
to 
\begin{align}
 \E \bigg[ \ft F_j(n, t_1)\cj{\ft F_j(n, t_2)}\bigg] 
& \sim
\sum_{n = n_1+ \cdots+n_{k_j}}
\E\bigg[
\Ff_j(n_1, \dots, n_{k_j}, t_1) \cj{\Ff(n_1, \dots, n_{k_j}, t_2)}\bigg].
\label{AZ} 
\end{align}

\smallskip

\noi
(iii)
In practice, in applying Proposition \ref{PROP:main}\,(ii)
to a product, 
we may assume that    $A = \{1\}$ and $B= \{2, \dots, J\}$
since, otherwise, i.e.~if $|A| \geq 2$, then
the product $F_{j} F_\l$, $j, \l \in A$ is divergent in the singular setting, 
requiring a further renormalization.
Even in the case $|A| \geq 2$, 
Proposition~\ref{PROP:main}\,(ii) is useful
in estimating non-standard products such as paraproducts
(without using the deterministic paraproduct estimate).

In Proposition \ref{PROP:sto1}\,(iii), 
we apply Proposition \ref{PROP:main}\,(ii) with $A = \{1\}$
to estimate  products of stochastic objects
for the dispersion-generalized nonlinear wave equations.
Our proof shows that as long as the lower
order terms,  belonging to $\H_{K_J - 2\l}$, $\l = 1, \dots, \big[\frac{K_J}{2}\big]$, 
are convergent, 
the regularity of the resulting stochastic term is
determined by the contribution
from the highest order term 
belonging to $\H_{K_J}$, 
which is estimated by 
Proposition \ref{PROP:main}\,(i).
See Remark~\ref{REM:finite}.

\smallskip

\noi
(iv) While the bounds \eqref{P3} and \eqref{P4}
look complicated, they are in their simplest form
and easy to use.
See the proof of Proposition \ref{PROP:sto1}\,(iii) below.
Without such  estimates, 
a concrete combinatorial argument can lead to 
a computational nightmare, even for a product 
with a relatively small number of terms.
See, for example, 
Appendix A in \cite{OTh1}
(which is more in the context of Proposition \ref{PROP:main}\,(i)).

 \end{remark}

\medskip
 
 We first present a proof of 
Proposition \ref{PROP:main}
in  Section \ref{SEC:2}.
In Section \ref{SEC:3}, we then 
consider 
 the dispersion-generalized nonlinear wave equations
 as an example,
 for which we construct 
 stochastic 
 objects, using Proposition \ref{PROP:main}, 
  and prove local well-posedness (Theorem \ref{THM:LWP}).
We point out that even with Proposition \ref{PROP:main} at hand, 
there is some nontrivial combinatorial difficulty that we need to overcome;
see the proof of Proposition \ref{PROP:sto1}.

 \section{Proof of Proposition \ref{PROP:main}}
\label{SEC:2}

We first recall the Wiener chaos estimate
(\cite[Theorem~I.22]{Simon})
which follows as 
 a consequence
of the  hypercontractivity of the Ornstein-Uhlenbeck semigroup
 due to Nelson \cite{Nelson2}. 
See also \cite[Proposition~2.4]{TTz}.

\begin{lemma}\label{LEM:hyp}
Let $k \in \N$.
Then, we have
\begin{equation*}
\|X \|_{L^p(\O)} \leq (p-1)^\frac{k}{2} \|X\|_{L^2(\O)}
 \end{equation*}
 
 \noi
 for any $p \geq 2$
 and any $X \in \H_{\leq k}$.

\end{lemma}

We now present the proof of Proposition \ref{PROP:main}.
 
\begin{proof}[Proof of Proposition \ref{PROP:main}] 
(i) 
We first carry out analysis for fixed $t  \in \R_+$
and thus suppress the $t$-dependence for notational simplicity.
Recall the definition of $K_j$ from \eqref{KK1}.
By 
\eqref{A1},  \eqref{A4}, 
Cauchy-Schwarz' inequality (in $\o$),  
and Cauchy's inequality
(together with the Pythagoras theorem to remove $\pi_{K_J}$), we have 
\begin{align}
 \E  \bigg[  & \Big| \pi_{K_J } \F_x\Big(\prod_{j = 1}^J F_j \Big)(n)\Big|^2\bigg] \notag \\
& = 
\sum_{n = n_1 + \cdots + n_{K_J}}
\sum_{n = m_1 + \cdots + m_{K_J}}
  \E  \bigg[ 
\pi_{K_J }   \bigg(\prod_{j = 1}^J
\Ff_j(n_{K_{j-1} + 1}, \dots, n_{K_j}) \bigg)\notag\\
& \hphantom{XXX} \times 
\pi_{K_J }  \bigg(\prod_{j = 1}^J\cj{\Ff_j(m_{K_{j-1}+1}, \dots, m_{K_j})}\bigg)
\bigg] \notag \\
& = 
\sum_{\s \in S_{K_J}}\, 
\sum_{n = n_1 + \cdots + n_{K_J}}
  \E  \bigg[  \pi_{K_J}  \bigg(\prod_{j = 1}^J
\Ff_j(n_{K_{j-1} + 1}, \dots, n_{K_j}) \bigg) \notag\\
& \hphantom{XXX} \times 
\pi_{K_J} \bigg(\prod_{j = 1}^J\cj{\Ff_j(n_{\s(K_{j-1}+1)}, \dots, n_{\s(K_j)})}\bigg)
\bigg] \notag\\
& \les 
\sum_{n = n_1+ \cdots+ n_{K_J}}
  \E  \bigg[    \Big|\prod_{j = 1}^J
\Ff_j(n_{K_{j-1} + 1}, \dots, n_{K_j}) \Big|^2 \bigg]\notag \\
& \hphantom{XXX} 
+ 
\sum_{\s \in S_{K_J}}\,
\sum_{n = n_1 +\cdots+ n_{K_J}}
\E\bigg[
\Big|\prod_{j = 1}^J\Ff_j(n_{\s(K_{j-1}+1)}, \dots, n_{\s(K_j)})\Big|^2
\bigg]\notag \\
& \les 
\sum_{n = n_1+ \cdots+ n_{K_J}}
  \E  \bigg[    \prod_{j = 1}^J
|\Ff_j(n_{K_{j-1} + 1}, \dots, n_{K_j})|^2 \bigg]
\label{B1}
 \end{align}

 \noi
 for any $n \in \Z^d$, 
 where the last step follows
 since  the action of the permutation $\s$ 
 amounts to relabeling  indices.
Then, from \eqref{B1}, H\"older's inequality, 
Minkowski's integral inequality, 
and the Wiener chaos estimate (Lemma \ref{LEM:hyp}),
we have
 \begin{align}
& \text{RHS of } \eqref{B1}\notag \\
&  \hphantom{X}
=
\sum_{n = m_1 + \cdots + m_J} 
\E\bigg[
\prod_{j = 1}^J 
\big\| \ind_{m_j = n_{K_{j-1} + 1}+\cdots + n_{K_j}}\notag \\
& \hphantom{XXXXXXXXXXX}
\times
 \Ff_j(n_{K_{j-1} + 1}, \dots, n_{K_j})\big\|_{\l^2_{n_{K_{j-1} + 1}, \dots, n_{K_j}}}^2 \bigg]\notag \\
&  \hphantom{X}  \leq 
\sum_{n = m_1 + \cdots + m_J}\prod_{j = 1}^J 
\bigg\| \big\|\ind_{m_j = n_{K_{j-1} + 1}+\cdots + n_{K_j}}\notag \\
& \hphantom{XXXXXXXXXXX}
\times
 \Ff_j(n_{K_{j-1} + 1}, \dots, n_{K_j})\big\|_{L^{2J}(\O)} \bigg\|_{\l^2_{n_{K_{j-1} + 1}, \dots, n_{K_j}}}^2
 \notag \\
&  \hphantom{X}  \les
\sum_{n = m_1 + \cdots + m_J}
\bigg( \prod_{j = 1}^J \ \sum_{m_j  = n_{K_{j-1} + 1}+\cdots + n_{K_j}}
\E\Big[ | \Ff_j(n_{K_{j-1} + 1}, \dots, n_{K_j})|^2\Big]\bigg).
\label{B2}
 \end{align}

\noi
Then,  from \eqref{AZ}, we obtain \eqref{P1}.

By a computation analogous to \eqref{B1} with  \eqref{diff}, 
we have 
 \begin{align}
  \E  &  \bigg[  \Big| \dl_h \pi_{K_J } \F_x\Big(\prod_{j = 1}^J F_j \Big)(n, t)\Big|^2\bigg]  \notag \\
& \les 
\sum_{n = n_1+ \cdots+ n_{K_J}}
  \E  \Bigg[  \bigg| \dl_h \Big( \prod_{j = 1}^J
\Ff_j(n_{K_{j-1} + 1}, \dots, n_{K_j}, t)\Big)\bigg|^2 \Bigg]\notag \\
& \les 
\sum_{\l = 1}^J
\, \sum_{n = n_1 + \cdots + n_{K_J}}
  \E  \Bigg[  \bigg| 
\dl_h \Ff_\l(n_{K_{\l-1} + 1}, \dots, n_{K_\l}, t)\notag  \\ 
& \hphantom{X}
\times  \bigg(  \prod_{j = 1}^{\l-1}
\Ff_j(n_{K_{j-1} + 1}, \dots, n_{K_j}, t+h)\bigg)
\bigg(  \prod_{j = \l+1}^J
\Ff_j(n_{K_{j-1} + 1}, \dots, n_{K_j}, t)\bigg)\bigg|^2\Bigg].
\label{B3}
 \end{align}

\noi
Then, by proceeding as in \eqref{B2} to bound
each term on the right-hand side of \eqref{B3}, 
we obtain~\eqref{P2}.

\smallskip

\noi
(ii) We only prove the bound \eqref{P3}.
As in Part (i), 
we  suppress the $t$-dependence.
By the orthogonality of $\H_k$ with  \eqref{A1}, \eqref{A5}, \eqref{A6}, and \eqref{A7}, we have
\begin{align}
\begin{split}
\E & \bigg[ \Big|\F_x\Big(\prod_{j = 1}^J F_j \Big)(n)\Big|^2\bigg]
= \sum_{\l = 0}^{[K_J/2]}
\E \bigg[ \Big| \pi_{K_J - 2\l} \F_x\Big(\prod_{j = 1}^J F_j \Big)(n)\Big|^2\bigg]\\
& = \sum_{\l = 0}^{[K_J/2]}
\sum_{n = n_1 + \cdots + n_{K_J}}
\sum_{n = m_1 + \cdots + m_{K_J}}
  \E  \bigg[ 
\pi_{K_J - 2\l}   \bigg(\prod_{j = 1}^J
\Ff_j(n_{K_{j-1} + 1}, \dots, n_{K_j}) \bigg)\\
& \quad \times 
\pi_{K_J - 2\l}  \bigg(\prod_{j = 1}^J\cj{\Ff_j(m_{K_{j-1}+1}, \dots, m_{K_j})}\bigg)
\bigg] .
\end{split}
\label{K1}
\end{align}

\noi
Given $\PP \in \Pi_\l$ for some $\l = 0, 1, \dots, \big[\frac{K_J}{2}\big]$, 
let  $S_{K_J - 2\l}(\PP)$ denote
the symmetric group on the unpaired indices $\{ i \in \big\{ 1,  \dots, K_J\}: i \notin \PP\big\}$.
Then, by 
\eqref{A9} and Cauchy's inequality, we have
\begin{align}
& \sum_{n = n_1 + \cdots + n_{K_J}}
\sum_{n = m_1 + \cdots + m_{K_J}}
\E   \bigg[ 
\pi_{K_J - 2\l}   \bigg(\prod_{j = 1}^J
\Ff_j(n_{K_{j-1} + 1}, \dots, n_{K_j}) \bigg)\notag \\
& \hphantom{XXX}\times 
\pi_{K_J - 2\l}  \bigg(\prod_{j = 1}^J\cj{\Ff_j(m_{K_{j-1}+1}, \dots, m_{K_j})}\bigg)
\bigg] \notag \\
& \hphantom{X}
=\sum_{\PP_1, \PP_2\in  \Pi_\l} \E \bigg[ 
 \bigg(  \sum_{n = n_1 + \cdots + n_{K_J}}\ind_{\substack{(n_1, \dots, n_{K_J})\\\text{admissible}\\
 \text{w.r.t.~} \PP_1}}
\, \Gf_{\PP_1}(n_1, \dots, n_{K_J})
\bigg)\notag \\
& \hphantom{XXX}\times 
   \bigg(
\sum_{n = m_1 + \cdots + m_{K_J}}
\ind_{\substack{(m_1, \dots, m_{K_J})\\\text{admissible}\\ \text{w.r.t.~} \PP_2}}
\cj{\Gf_{\PP_2}(m_1, \dots, m_{K_J})}
\bigg) \bigg]\notag \\
& \hphantom{X}
 \les  \sum_{\PP\in  \Pi_\l} 
\,  \sum_{\s \in S_{K_J - 2\l}(\PP)}
\E \bigg[ 
\sum_{\{ n_i, m_i \}_{i \notin \PP}}  \ind_{\substack{m_i = n_{\s(i)}\\i \notin \PP}}\notag \\
& \hphantom{XXX}\times 
 \bigg(  \sum_{\{n_i \}_{i \in\PP}} \ind_{n = n_1 + \cdots + n_{K_J}}\ind_{\substack{(n_1, \dots, n_{K_J})\\\text{admissible}\\ \text{w.r.t.~} \PP}}
\, \Gf_{\PP}(n_1, \dots, n_{K_J})
\bigg)\notag \\
& \hphantom{XXX}\times 
 \bigg( 
  \sum_{\{m_i \}_{i \in\PP}}
  \ind_{n = m_1 + \cdots + m_{K_J}}\ind_{\substack{(m_1, \dots, m_{K_J})\\\text{admissible}\\ \text{w.r.t.~} \PP}}
\, \cj{\Gf_{\PP}(m_1, \dots, m_{K_J})}
\bigg)\bigg]\notag \\
& \hphantom{X}
 \les  \sum_{\PP\in  \Pi_\l} 
 \sum_{\{n_i\}_{i\notin \PP}}
 \E \Bigg[ 
 \bigg|  
  \sum_{\{n_i\}_{i\in \PP}}
\ind_{n = n_1 + \cdots + n_{K_J}}\ind_{\substack{(n_1, \dots, n_{K_J})\\\text{admissible}\\ \text{w.r.t.~} \PP}}
\, \Gf_{\PP}(n_1, \dots, n_{K_J})
\bigg|^2 \Bigg]. 
\label{K2}
\end{align}

\noi
Here, the first inequality (i.e.~the second step) in \eqref{K2} follows
from \eqref{A10} and the definition of $\{\gf_n\}_{n \in \Z^d}$.
More precisely, 
noting that 
$\sum_{j = 1}^J |N_j \setminus \PP| = K_J - 2\l$, 
it follows from the definition of 
$\{\gf_n\}_{n \in \Z^d}$ and  Wick renormalized products
that 
\begin{align*}
\E& \Bigg[\pi_{K_J - 2\l}\bigg( \bigg( 
\prod_{j \in B} \Big(\prod_{i \in N_j \setminus \PP}
\gf_{n_i}(\tau_j)\Big)\bigg)
\bigg(
 \prod_{i \in \bigcup_{j \in A} N_j \setminus \PP}
\gf_{n_i}(t)\bigg)\bigg)\\
& \times \pi_{K_J - 2\l}\bigg(\cj{ \bigg( 
\prod_{j \in B} \Big(\prod_{i \in N_j \setminus \PP}
\gf_{m_i}(\tau_j)\Big)\bigg)
\bigg(
 \prod_{i \in \bigcup_{j \in A} N_j \setminus \PP}
\gf_{m_i}(t)\bigg)}\bigg)\Bigg]
= 0
\end{align*}

\noi
unless
$\{ n_i \}_{i \notin \PP} = 
\{ m_i \}_{i \notin \PP}$.

Therefore, the bound \eqref{P3} follows from \eqref{K1} and \eqref{K2}.
\end{proof}

\section{Example:
 dispersion-generalized nonlinear wave equations}
\label{SEC:3}

\subsection{Dispersion-generalized nonlinear wave equation with random initial data}

For $\al > 0$, consider the  following dispersion-generalized nonlinear wave equation
(NLW)
on $\T^d$:
\begin{align}
\begin{cases}
\dt^2 u + (1-\Dl)^\al u + u^k = 0\\
(u, \dt u )|_{t = 0} = (u_0, u_1)
\end{cases}
\label{NLW1}
\end{align}

\noi
with the random initial data of the form
\begin{equation}\label{series}
u_0^\o = \sum_{n \in \Z^d} \frac{g_n(\o)}{\jb{n}^{\be}}e^{in\cdot x}
\qquad\text{and}\qquad 
u_1^\o = \sum_{n \in \Z^d} \frac{h_n(\o)}{\jb{n}^{\be-\al}}e^{in\cdot x}
\end{equation}

\noi
for $\be \in \R$, 
where  the series   $\{ g_n ,  h_n \}_{n \in \Z^d}$ is a family of 
independent standard   complex-valued  Gaussian random variables  
conditioned that  $g_{-n}=\cj{g_{n}}$ and  $h_{-n}=\cj{h_{n}}$, 
$n \in \Z^d$. 
It is easy to see that 
$u_0^\o$ in \eqref{series} belongs
to $ W^{s, p}(\T^d) \setminus  W^{\be - \frac d2, p}(\T^d)$
for any $s < \be - \frac d2$ and $1\le p \le \infty$, 
almost surely.
In the following, we consider the singular  case $\be \le \frac d 2$
such that $u_0^\o$ is merely a distribution.

Let $Z$ denote the random linear solution given by 
\begin{align}
\begin{split}
Z(t) & = \cos (t \jb{\nb}^\al) u_0^\o + \frac{\sin (t \jb{\nb}^\al)}{\jb{\nb}^\al} u_1^\o\\
& = \sum_{n \in \Z^d} \frac{\cos (t \jb{n}^\al) g_n(\o)
+ \sin (t \jb{n}^\al)h_n(\o)}{\jb{n}^{\be}}e^{in\cdot x}, 
\end{split}
\label{NLW2}
\end{align}

\noi
where $\jb{\nb}^\al = (1-\Dl)^\frac{\al}{2}$.
The case  $\be = \al$ is of particular importance 
since 
the distribution of $(u_0^\o, u_1^\o)$ in \eqref{series} 
corresponds to  the base Gaussian measure
of the Gibbs measure for the dispersion-generalized
NLW~\eqref{NLW1}. 
See,  for example, 
 \cite{OTh2, STzW, FPT}
when  $d = 2$ and $\frac 12 < \al = \be \le 1$.

Given $N \in \N$, we set $Z_N = \P_N Z$, 
where $\P_N$ denotes the frequency cutoff onto the spatial frequencies $\{ |n| \leq N \}$. Then, 
for each fixed  $x \in \T^d$ and $t \geq 0$, 
a direct computation with~\eqref{NLW2} shows that
 $Z_N (x, t)$ is a mean-zero real-valued Gaussian random variable with variance
\begin{align*}
\s_N = \E\big[ (Z_N (x, t))^2 \big] = \sum_{|n| \le N} \frac{1}{\jb{n}^{2\be}} \too \infty, 
\end{align*}

\noi
as $N \to \infty$, since $\be \le \frac d 2$.
This essentially shows that 
 $\{Z_N(t)\}_{N \in \N}$
is almost surely unbounded in $W^{0, p}(\T^d)$ for any $1 \leq p \leq \infty$.
We now introduce  the Wick renormalization:
\begin{equation}
\wick{Z_N^\ell (x, t)} \, \deff H_\ell (Z_N(x, t); \s_N),
\label{Wick1}
\end{equation}

\noi
where $H_\ell (x,\s)$ is the Hermite polynomial of degree $\ell$ with a variance parameter $\s$, 
defined by the generating function: 
\begin{equation}
G(t, x; \s) = e^{tx - \frac{1}{2}\s t^2} = \sum_{k = 0}^\infty \frac{t^k}{k!} H_k(x;\s).
\label{Wick1a}
 \end{equation}

\noi
Next, we define the second order term
$ \I( :\! \hspace{-0.6mm}Z_N^k \hspace{-0.6mm} \!:)$, where $\I$ denotes the Duhamel integral operator given by 
\begin{align*}
 \I(F) (t) = \int_0^t \frac{\sin((t-\tau)\jb{\nb}^\al)}{\jb{\nb}^\al} 
F(\tau) d\tau.
\end{align*}

As in \cite{GKO, GKOT, OPTz, OOcomp}, 
we consider the renormalized version of \eqref{NLW1}
with the truncated initial data $(\P_N u_0^\o, \P_Nu_1^\o) \in C^\infty(\T^d) \times C^\infty(\T^d)$.
Let $u_N$ be the solution to \eqref{NLW1} with the truncated
initial data
$(u_N, \dt u_N) |_{t = 0} = (\P_N u_0^\o, \P_Nu_1^\o)$.
Then, by writing $u_N$ in the first order expansion $u_N = Z_N + w_N$, 
we have
\begin{align*}
u_N^k = \sum_{\ell=0}^k {k\choose \ell} Z_N^\ell w_N^{k-\ell}.
\end{align*}

\noi
In the singular setting, namely, $\be \le \frac d2$, 
$Z = \lim_{N \to \infty} Z_N$ is only a distribution
and thus the limit of $Z_N^\l$ does not exist.
In order to overcome this issue, we consider the following renormalized nonlinearity:
\begin{align*}
:\!u_N^k\!: \, \stackrel{\text{def}}{=}  H_k( Z_N + w_N; \s_N) = \sum_{\ell=0}^k {k\choose \ell} :\!Z_N^\ell \!: w_N^{k-\ell},
\end{align*}

\noi
where the right-hand side is well defined as long as $w_N$
has sufficient regularity.
In terms of the remainder term $w_N = u_N - Z_N$, 
the renormalized equation reads as follows:
\begin{align}
\dt^2 w_N &  + (1-\Dl)^\al w_N  
 + \sum_{\l = 0}^{k}  {k\choose \ell} :\!Z_N^\ell \!: w_N^{k-\ell}
= 0
\label{NLW2a}
\end{align}

\noi
\noi
with the zero initial data
$(w_N, \dt w_N )|_{t = 0} = (0, 0)$.
The spatial regularity of $Z_N^\l$ (in the limiting sense as $N \to \infty$)
is given by $\l (\be - \frac d2) - \eps$ for any $\eps > 0$ (see Proposition \ref{PROP:sto1}\,(i) below)
and in particular, when $\be < \frac d2$, 
the worst term in the nonlinearity is given by
the highest order term $:\!Z^k\!:$.
In the next step, we consider the second order expansion and remove this worst term.

Write $u_N$ in the second order expansion
\begin{align}
 u_N = Z_N -  \I(:\!Z_N^k\!:) + v_N,\quad \text{namely,}
\quad  w_N = -  \I(:\!Z_N^k\!:) + v_N,
\label{exp2}
\end{align}

\noi
where $w_N$ satisfies \eqref{NLW2a}.
Then,  by noting that 
\begin{align}
(\dt + (1 - \Dl)^\al)  \I(:\!Z_N^k\!:) = :\!Z_N^k\!:,
\label{NLW2b}
\end{align}

\noi
 we can reduce~\eqref{NLW2a}
to the following  renormalized equation 
for the new remainder term $v_N$:
\begin{align}
\begin{split}
\dt^2 v_N &  + (1-\Dl)^\al v_N  \\
& + \sum_{k_1 = 0}^{k-1} 
\sum_{k_2, k_3 = 0}^k
\ind_{k = k_1 + k_2 + k_3} \frac{k!}{k_1!k_2!k_3!}
:\!Z_N^{k_1}\!:  \big( \I(:\!Z_N^k\!:)\big)^{k_2} v_N^{k_3}
= 0
\end{split}
\label{NLW3}
\end{align}

\noi
with the zero initial data
$(v_N, \dt v_N )|_{t = 0} = (0, 0)$.
Note that, thanks to the second order expansion (in particular, \eqref{NLW2b}), 
the worst term  
$:\!Z_N^{k}\!:$ in the equation \eqref{NLW2a} for 
the remainder term $w_N$ of the first order expansion is eliminated.
See also Subsection 1.2 in \cite{OPTz}.
In order to study the well-posedness issue of~\eqref{NLW3}, 
 we first need to study the regularity property
of the stochastic objects
\begin{align}
:\!Z_N^{k_1}\!:  \big(\I(:\!Z_N^k\!:)\big)^{k_2}
\label{sto1}
\end{align}

\noi
for 
$k_1 = 0, \dots, k-1$ and $k_2 = 0, \dots, k$.
In the next subsection,  we study the regularity property of 
these stochastic objects by using 
 Proposition \ref{PROP:main}
 and Lemma \ref{LEM:reg}.
 See Proposition~\ref{PROP:sto1}.
In Subsection \ref{SUBSEC:3.3}, 
we then establish  
local well-posedness of \eqref{NLW3};
see Theorem \ref{THM:LWP} below.

\subsection{Regularities of stochastic objects}
Our main goal in this subsection is to prove the following
proposition on the regularity property
of the stochastic terms in \eqref{sto1}, 
appearing in the equation \eqref{NLW3}.

\begin{proposition}\label{PROP:sto1}

Let $\al > 0$,  $\be \le \frac d2$, and $k \in \N$.\rule[-2mm]{0mm}{0mm}
\\
%
\noi
\textup{(i)}
Given  $\l \in \N$, let   $\be > \frac{\l-1}{2\l}d$.
Then,  for  $s < \l(\be - \frac d2)$, 
$\{:\!Z_N^\l\!: \}_{N \in \N}$ is 
 a Cauchy sequence
in 
$C(\R_+;W^{s,\infty}(\T^d))$,  almost surely.
In particular,
denoting the limit by $:\!Z^\l\!: $,  
we have
  \[:\!Z^\l\!: \, \in C(\R_+;W^{\l(\be - \frac d2)-\eps, \infty}(\T^d))
  \]
  
  \noi 
for any $\eps >0$, almost surely.

\smallskip

\noi
\textup{(ii)}
Assume $\be > \frac{k-1}{2k}d$.
Then, for   $s < k(\be - \frac d2)+\al$, 
$\{  \I(:\!Z_N^k\!:)  \}_{N \in \N}$ is 
 a Cauchy sequence
in 
$C(\R_+;W^{s,\infty}(\T^d))$,  almost surely.
In particular,
denoting the limit by $ \I(:\!Z^k\!:) $, 
we have
  \[ \I(:\!Z^k\!:)  \in C(\R_+;W^{k(\be - \frac d2)+\al-\eps, \infty}(\T^d))
  \]
  
  \noi 
for any $\eps >0$, almost surely.

Furthermore, suppose that 
$k(\be - \frac d2)+\al >0 $.
Let $k_2 \in \N$.
Then, for   $s < k(\be - \frac d2)+\al$, 
$\big\{  \big(\I(:\!Z_N^k\!:)  \big)^{k_2} \big\}_{N \in \N}$ is 
 a Cauchy sequence
in 
$C(\R_+;W^{s,\infty}(\T^d))$,  almost surely.
In particular,
denoting the limit by $\big( \I(:\!Z^k\!:) \big)^{k_2}$, 
we have
  \[ \big(\I(:\!Z^k\!:) \big)^{k_2} \in C(\R_+;W^{k(\be - \frac d2)+\al-\eps, \infty}(\T^d))
  \]
  
  \noi 
for any $\eps >0$, almost surely.

\smallskip
\noi
\textup{(iii)} 
Fix  integers $0 \le k_1 \le k-1$ and $0 \le k_2 \le k$.
Suppose that  $\be > \frac{k-1}{2k}d$ and 
 $k(\be - \frac d2)+\al> 0$, 
 which are the conditions from \textup{(i)} and \textup{(ii)}.
Furthermore, we assume that 
\begin{align}
\be > \frac d2 - \frac{\al}{2k}. 
\label{Zx1}
\end{align}

\noi
Given $N \in \N$, define $Y_N$ by 
\[Y_N =\, :\!Z_N^{k_1}\!:  \big(\I(:\!Z_N^k\!:)\big)^{k_2}.\]

\noi
Then,  for  $s < k_1(\be - \frac d2)$, 
$\{Y_N \}_{N \in \N}$ is 
 a Cauchy sequence
in 
$C(\R_+;W^{s,\infty}(\T^d))$,  almost surely.
In particular,
denoting the limit by $ 
Y =\, :\!Z^{k_1}\!:  \big(\I(:\!Z^k\!:)\big)^{k_2} $, 
we have
  \[ Y = :\!Z^{k_1}\!:  \big(\I(:\!Z^k\!:)\big)^{k_2}  \in C(\R_+;W^{k_1(\be - \frac d2)-\eps, \infty}(\T^d))
  \]
  
  \noi 
for any $\eps >0$, almost surely.

\end{proposition}

By putting all the conditions together, 
the range of admissible $\be$ is given by 
\[ \max \bigg(\frac d 2 - \frac{d}{2k}, \frac d2 - \frac \al {2k}\bigg)
< \be \le \frac d2 , \]

\noi
which gets smaller and smaller 
 as $\al \to 0$ and $k \to \infty$.

It is possible to exploit multilinear dispersive smoothing 
to improve Part (ii) (and hence Part~(iii)) of Proposition  \ref{PROP:sto1}
(especially when $\al > 0$ is not small).
See for example \cite{GKO2, Bring, OWZ}.
Note that such multilinear dispersive analysis depends
sensitively on the values of $k$ and $\al$ and moreover
the gain from multilinear dispersive smoothing becomes small
for small values of $\al$ (i.e.~in a weakly dispersive case).
We do not pursue this direction in this note.

Before proceeding to the proof of Proposition \ref{PROP:sto1}, 
we first recall a calculus lemma.

\begin{lemma}\label{LEM:SUM}
\textup{(i)}
Let $d \geq 1$ and $\al, \be \in \R$ satisfy

\[ \al+ \be > d  \qquad \text{and}\qquad \al, \be < d.\]
\noi
Then, we have
\[
 \sum_{n = n_1 + n_2} \frac{1}{\jb{n_1}^\al \jb{n_2}^\be}
\les \jb{n}^{d - \al - \be}\]

\noi
for any $n \in \Z^d$.

\smallskip

\noi
\textup{(ii)} 
Let $d \geq 1$ and $\al, \be, \g \in \R$ satisfy

\[ \al+ 2\be > d, \qquad \al + 2 \g > d,  \qquad \text{and}\qquad \al, 2\be, 2\g < d.\]
\noi
Then, we have
\[
 \sum_{n \in \Z^d} \frac{1}{\jb{n}^\al \jb{n + a}^\be \jb{n + b}^\g   }
\les \jb{a}^{\frac{d}{2} - \frac{\al}{2} - \be} \jb{b}^{\frac{d}{2} - \frac{\al}{2} - \g} \]

\noi
for any $a,b \in \R^d$.

\end{lemma}

Lemma \ref{LEM:SUM} follows
from elementary  computations.
See, for example,  Lemma 4.1  in \cite{MWX} for the proof of (i). 
Then,  (ii) is a straightforward consequence of (i) and Cauchy-Schwarz' inequality.

\begin{proof}[Proof of Proposition \ref{PROP:sto1}]
(i) 
Recalling that $:\!Z_N^\l\!: = \pi_\l (Z_N^\l)$, 
Proposition \ref{PROP:main}\,(i) yields
\begin{align*}
\E\Big[|\ft{:\!Z_N^\l\!:} (n, t)|^2\Big]
& \les
\sum_{n = n_1 + \dots + n_\l}
\prod_{j = 1}^\l \E\Big[|\ft Z_N (n_j, t)|^2\Big]\\
&  \les \sum_{\substack{n = n_1+ \cdots+  n_\l\\|n_j|\le N}}
\prod_{j = 1}^\l \frac{1}{\jb{n_j}^{2\be}}.
\end{align*}

\noi
Then, by iteratively carrying out the summations via Lemma \ref{LEM:SUM}\,(i), 
we obtain
\begin{align}
\E\Big[|\ft{:\!Z_N^\l\!:} (n, t)|^2\Big]
& \les \jb{n}^{(\l-1) d - 2\l\be +\eps_0}
= \jb{n}^{-d - 2\l(\be-\frac d2) + \eps_0}
\label{Z0a}
\end{align}

\noi
with any small $\eps_0 >0$ when $\be = \frac d2$ and $\eps_0 = 0$ when $\be < \frac d2$, 
provided that $\be > \frac{\l-1}{2\l}d$.
A slight modification of the computation above yields
\begin{align*}
\E\Big[|\ft{:\!Z_N^\l\!:} (n, t) - \ft{:\!Z_M^\l\!:} (n, t)|^2\Big]
& \les N^{-\g} \jb{n}^{-d - 2\l(\be-\frac d2) + \g + \eps_0}
\end{align*}

\noi
for some small $\g > 0$ and any $ M \ge N \ge 1$.
Hence, from  Lemma~\ref{LEM:reg}, we
see that $:\!Z_N^\l(t)\!:$ converges to some limit in  $W^{s, \infty}(\T^d)$
for $s < \l(\be - \frac d2)$, almost surely.
A similar computation with the mean value theorem 
gives difference estimates \eqref{reg2} and \eqref{reg3}, 
from which we conclude that 
 $:\!Z_N^\l\!:$ converges to some limit in  $C(\R_+; W^{s, \infty}(\T^d))$
for $s < \l(\be - \frac d2)$, almost surely.

In the remaining part of the proof, 
we only establish uniform (in $N$) bounds 
for fixed $t \in \R_+$ on relevant stochastic objects.
A slight modification yields the continuity-in-time and convergence claims.
For simplicity of notation, we drop the subscript $N$.
Due to the presence of the time integration in the Duhamel integral operator, 
various estimates depend on the time $t$.
Since such $t$-dependence plays no important role, 
 we hide the  $t$-dependence in implicit constants, 
appearing in the estimates.

\smallskip

\noi
(ii) 
The first claim 
follows easily from  Part (i) with $\l = k$
and the gain of $\al$-derivative under the Duhamel integral operator $\I$.
In particular, from \eqref{Z0a}, we have
\begin{align}
\E\Big[|\ft{\I(:\!Z^k\!:) } (n, t)|^2\Big]
& \les 
 \jb{n}^{-d - 2(k(\be-\frac d2)+\al) + \eps_0}
\label{Z0b}
\end{align}

\noi
with any small $\eps_0 >0$ when $\be = \frac d2$ and $\eps_0 = 0$ when $\be < \frac d2$, 
provided that $\be > \frac{k-1}{2k}d$.
By Lemma \ref{LEM:reg}, 
we conclude that 
$\I(:\!Z^k\!:) (t) \in W^{s, \infty}(\T^d)$ for any $s < k(\be - \frac d2)+\al$, almost surely.

Now, suppose that $k(\be - \frac d2)+\al > 0$.
Then, there exist $s>0$ and $ \eps > 0$ with $0 < s < s +\eps < k(\be - \frac d2)+\al $
such that 
$\I(:\!Z^k\!:) (t) \in W^{s+\eps, \infty}(\T^d)$, almost surely.
By Sobolev's inequality (with $\eps r > d$) and the fractional Leibniz rule 
(see Lemma \ref{LEM:bilin}\,(i) below), 
we have  
\begin{align*}
\big\| \big(\I(:\!Z^k\!:)\big)^{k_2} (t)\big\|_{W^{s, \infty}}
& \les \big\| \big(\I(:\!Z^k\!:)\big)^{k_2} (t)\big\|_{W^{s+\eps, r}}
\les \| \I(:\!Z^k\!:) (t)\|_{W^{s+\eps, \infty}}^{k_2}\\
& < \infty, 
\end{align*}

\noi
almost surely.  This proves the second claim.

\smallskip

\noi
(iii)
From \eqref{Wick1}
with \eqref{NLW2}, 
we have 
$:\!Z^{k_1}(t) \!: \,  \in \H_{k_1}$
and 
$\I(:\!Z^k\!:)(t) \in \H_{k}$.
Thus, we have 
$Y(t)  = \, \wick{Z^{k_1}}
\big(\I(:\!Z^k\!:)\big)^{k_2}(t) \in \H_{\le k_1 + k k_2}$.
We first estimate the contribution 
of $Y(t)$ belonging to $\H_{ k_1 + k k_2}$.
By Proposition \ref{PROP:main}\,(i) 
with \eqref{Z0a},  \eqref{Z0b}, 
and Lemma \ref{LEM:SUM}\,(i), we have 
\begin{align*}
 \E \Big[ |   \pi_{k_1 + k k_2} & \ft {Y}(n, t)|^2\Big]
 \les 
\sum_{n = n_1 + \cdots + n_{k_2 + 1}} 
\E \Big[ |\ft{\wick{Z^{k_1}}}(n_1, t)|^2\Big] 
\prod_{j = 2}^{k_2 + 1} \E \Big[ |\ft{\I(:\!Z^k\!:)}(n_j, t)|^2\Big]\\
& \les
\sum_{n = n_1 + \cdots + n_{k_2 + 1}} 
\jb{n_1}^{-d - 2k_1(\be-\frac d2) + \eps_0}
\prod_{j = 2}^{k_2 + 1} \jb{n_j}^{-d - 2(k(\be-\frac d2)+\al) + \eps_0}\\
& \les  
\jb{n}^{-d - 2k_1(\be-\frac d2) + \eps_1}
\end{align*}

\noi
for any small $\eps_1 > 0$, 
provided that $\be > \frac{k_1-1}{2k_1}d$ and  $k(\be - \frac d2)+\al> 0$.
Hence, by Lemma \ref{LEM:reg}, 
we have
$\pi_{k_1 + k k_2} Y(t) \in W^{s+\eps, \infty}(\T^d)$
for any $s <  k_1(\be - \frac d2)$.

Next, we use Proposition \ref{PROP:main}\,(ii) to estimate the contribution
of $Y(t)$
belonging to the lower order homogeneous Wiener chaoses.
In the following, we first describe various
parameters and objects, appearing  in 
Proposition \ref{PROP:main}\,(ii), 
adapted to our current setting.
Fix $\l = 1, \dots, \big[\frac{k_1 + k k_2}2\big]$.
With $J =  k_2 + 1$, set
\[K_0 = 0, \quad K_1 = k_1, 
\quad \text{and} \quad 
K_j = k_1 + (j-1) k, \quad j = 2, \dots, J, 
\]

\noi
and set $N_j = \{K_{j-1} + 1, \dots, K_j\}$.
In particular, we have $K_J = k_1 + k k_2$.
We set $A = \{1\}$ and $B = \{2, \dots, k_2 + 1\}$.
In view of \eqref{NLW2}, define $\gf_n(t)$ as in 
\eqref{NLW0}:
\[ \gf_n(t) = \cos (t \jb{n}^\al) g_n+ \sin (t \jb{n}^\al)h_n.\]

\noi
We also set 
\begin{align*}
S_{\nf_j} (t)  = \frac{\sin(t\jb{\nf_j}^\al)}{\jb{\nf_j}^\al}
\qquad \text{and} \qquad 
 C_j(n_{K_{j-1} + 1}, \dots, n_{K_j})  = 
\prod_{i \in N_j } \frac{1}{\jb{n_i}^\be}, 
\end{align*}

\noi
where
\begin{align}
\nf_j = \sum_{i \in N_j} n_i = n_{K_{j-1} + 1} + \cdots +  n_{K_j}.
\label{Z2}
\end{align}

\noi
In this setting, the deterministic function $Q_\PP$
in \eqref{A10} and \eqref{P4}
satisfies
\[  |Q_\PP(\{n_i, i \in \PP\},  \{\tau_j, j \in B\}, t)|\les 1, \]

\noi
uniformly in $\PP\in \Pi_\l$,
$\{n_i, i \in \PP\}$,  $\{\tau_j, j \in B\}$, $t \in \R_+$, 
since $Q_\PP$ is given by (the product of) the second moment
of products of Gaussians $g_{n_i}$ and $h_{n_i}$
multiplied by the trigonometric functions
$\cos (\tau \jb{n_i}^\al)$ and 
$\sin (\tau \jb{n_i}^\al)$ with $\tau = t$ or $\tau_j$.
Then, by Proposition \ref{PROP:main}\,(ii) with 
$\mf_j = m_{K_{j-1} + 1} + \cdots +  m_{K_j}$, 
we have 
\begin{align}
\begin{split}
 \E & \Big[ | \pi_{k_1 + k k_2 - 2\l} \ft {Y}(n, t)|^2\Big] \\
&   \les   \sup_{\PP\in  \Pi_\l} 
 \sum_{\{n_i\}_{i\notin \PP}}
 \Bigg\{
  \sum_{\{n_i\}_{i\in \PP}}
 \sum_{\{m_i\}_{i\in \PP}}
\ind_{\substack{n = n_1 + \cdots + n_{K_J}\\n = m_1 + \cdots + m_{K_J}}}
\ind_{\substack{(n_1, \dots, n_{K_J})\\\textup{admissible}\\ \textup{w.r.t.~} \PP}} 
\ind_{\substack{(m_1, \dots, m_{K_J})\\\textup{admissible}\\ \textup{w.r.t.~} \PP}} \\
& \quad \times 
\ind_{\substack{n_i = m_i\\i \notin \PP}}
\bigg(\prod_{j  = 2}^{k_2 + 1}
\frac{1}{\jb{\nf_j}^\al} \bigg)
\bigg(\prod_{i = 1}^{K_J} \frac{1}{\jb{n_i}^\be}\bigg)
\bigg(\prod_{j  = 2}^{k_2 + 1}
\frac{1}{\jb{\mf_j}^\al} \bigg)
\bigg(\prod_{i = 1}^{K_J} \frac{1}{\jb{m_i}^\be}\bigg)\Bigg\}.
\end{split}
\label{Z3}
\end{align}

 \noi
 Note that 
the condition $n_i = m_i$, $i \notin \PP$, in \eqref{Z3} follows from 
computing the expectation of
\begin{align*}
& \pi_{K_J - 2\l}\bigg( \bigg( 
\prod_{j \in B} \Big(\prod_{i \in N_j \setminus \PP}
\gf_{n_i}(\tau_j)\Big)\bigg)
\bigg(
 \prod_{i \in \bigcup_{j \in A} N_j \setminus \PP}
\gf_{n_i}(t)\bigg)\bigg)\\
& \hphantom{X}
\times 
\pi_{K_J - 2\l}\cj{ \bigg( \bigg( 
\prod_{j \in B} \Big(\prod_{i \in N_j \setminus \PP}
\gf_{m_i}(\tau_j)\Big)\bigg)
\bigg(
 \prod_{i \in \bigcup_{j \in A} N_j \setminus \PP}
\gf_{m_i}(t)\bigg)\bigg)}
\end{align*}

\noi
(which comes from 
\eqref{P4}), for example by using  Proposition \ref{PROP:main}\,(i)
(note that  $\sum_{j = 1}^J |N_j \setminus \PP| = K_J - 2\l$). 

In order to estimate \eqref{Z3}, 
we first need to carry out summations
over   $\{n_i\}_{i\in \PP}$
and ${\{m_i\}_{i\in \PP}}$,
which can be a priori divergent.
The outside summation over 
$\{n_i\}_{i\notin \PP}$ behaves better, 
thanks to the squared power $\jb{n_i}^{-2\be}$
(as compared to $\jb{n_i}^{-\be}$ for $\{n_i\}_{i\in \PP}$).
What comes in rescue
in summing over the paired frequencies, say in 
   $\{n_i\}_{i\in \PP}$, 
is the gain of derivatives 
$\jb{\nf_j}^{-\al}$ coming from the Duhamel integral operator.

Following Definition \ref{DEF:pair}, we write $\mathcal P$ as a disjoint union $\mathcal P = \mathcal P_1 \cup \PP_2$ with \[\mathcal P_1 \subset \bigcup_{j \in B} N_j= \bigcup_{j = 2}^{k_2+ 1} N_j\]

\noi
 such that an element $i \in \mathcal P_1$ is paired with exactly one element $i_0 \in \mathcal P_2$. 
Note that the choice of such $\PP_1$ (and thus $\PP_2$) is not unique but that 
any such partition $\PP = \PP_1 \cup \PP_2$ works equally. 
Given $2 \le j \le k_2 +1$, we define the following sets:
\begin{align}\label{Z31}
\begin{split}
& \mathcal P_1(N_j)  \deff \{ i \in \mathcal P_1 : (i,i_0) \in \mathcal P \text{ for some } i_0 \in \mathcal P_2 \cap N_j  \},  \\
& \mathcal P_2(N_j)  \deff \{ i \in \mathcal P_2 : (i,i_0) \in \mathcal P \text{ for some } i \in \mathcal P_1 \cap N_j  \}.
\end{split}
\end{align}

\noi
With these notations (and recalling that $\PP$ is admissible), we have
\begin{align}
\eqref{Z3} = 
\sup_{\PP \in \Pi_\l}
\sum_{ \{ n_i \}_{i  \not \in \mathcal P} } 
\bigg(\prod_{i \not \in \mathcal P} \frac{1}{ \jb{n_i}^{2 \be}} \bigg)
\ind_{n = \sum_{ i \not \in \mathcal P } n_i } \,
\big\{ S(\{n_i\}_{i \notin\PP})\big\}^2,
\label{Z32}
\end{align}

\noi
where $S(\{n_i\}_{i \notin\PP})$ is defined by 
\begin{align}
S(\{n_i\}_{i \notin\PP}) = \sum_{ \{ n_i \}_{i  \in \mathcal P_1 }}
\bigg( \prod_{i \in \mathcal P_1} \frac{1}{ \jb{n_i}^{2 \be}}\bigg) \prod_{j = 2}^{k_2 +1} \frac{1}{\jb{ \bar{ \nf }_j }^{\al}}.
\label{Z33}
\end{align}

\noi
Here, in view of \eqref{Z2}, \eqref{Z33}, and the definition of admissibility
in Definition \ref{DEF:pair}, 
$\bar \nf_j$ is given by 
\begin{align*}
 \bar{ \nf }_j 
 &  = \sum_{ i \in \PP^c \cap N_j  } n_i + \sum_{i \in \mathcal P_1 \cap N_j} n_i
 + \sum_{i \in \mathcal P_2 \cap N_j} n_i\\
 &  = \sum_{ i \in \PP^c \cap N_j  } n_i + \sum_{i \in \mathcal P_1 \cap N_j} n_i 
 - \sum_{\substack{ j_0 = 2 \\ j_0 \neq j }}^{ k_2 +1} \, \sum_{i \in \PP_1(N_j) \cap  N_{j_0} } n_i
\end{align*}

\noi
 for $2 \le j \le k_2 +1$. 
  The following simple observation plays a crucial role. 
First, recall that, in our current setting,  
we have $A = \{1\}$. 
Namely, the frequency $n_i$, $i \in N_j$, does not involve
the action of the Duhamel integral operators
if and only if $j = 1$.
Then,   
in view of  the definition~\eqref{Z31}, 
we see that   
 if $i \in \PP_1(N_1)$, then $i_0 \in \PP_2 \cap N_1$
 and thus the frequency $n_{i_0} = -n_i$
 does not appear in $\bar{\nf}_j $ for any $j = 2, \dots, k_2 + 1$.
 Hence, 
 if $i \in \PP_1(N_1)$, we conclude that the frequency $n_i$ 
 appears in exactly one $\bar{\nf}_j $.
 On the other hand,  if $i \in \PP_1(N_{j_0})$ for some $2 \le j_0 \le k_2 + 1$, then 
 $\pm n_i$ appears in $\bar{\nf}_j $ for exactly two values of $j = 2, \dots, k_2 + 1$.

In estimating the right-hand side of \eqref{Z32},
for simplicity of the presentation,  
we drop the supremum in $\PP \in \Pi_\l$  with the understanding that 
the estimates are uniform in $\PP\in \Pi_\l$.
 In the following, $\eps > 0$ denotes a sufficiently small positive constant, 
 which may be different line by line.

 \medskip
 
 \noi
 $\bullet$ {\bf Step 1:}
In this step, we sum 
over the variables $\{ n_i \}_{i \in \PP_1(N_1) }$.
By  Lemma \ref{LEM:SUM}\,(i), we have
\begin{align}
S( \{n_i\}_{i \notin \PP})
 \les \sum_{ \{ n_i \}_{i   \in \mathcal P_1\setminus \PP_1(N_1) } } 
 \bigg( \prod_{i \in \mathcal P_1 \setminus \PP_1(N_1) } \frac{1}{ \jb{n_i}^{2 \be}} \bigg)
 \prod_{j = 2}^{k_2 +1} \frac{1}{\jb{  \nf^{(1)}_j }^{\al^{(1)}_j}},
\label{Z35}
\end{align}

\noi
provided that  $\be > \frac{d}{2} - \frac{\al}{2k}$, 
where $\nf^{(1)}_j$ and $\al^{(1)}_j$, $j = 2, \dots, k_2+1$, are given by 
\noi
\begin{align*}
& \nf^{(1)}_j  = \sum_{ i \in \PP^c \cap N_j  } n_i + \sum_{i \in (\PP_1 \setminus \PP_1(N_1)) \cap N_j} n_i - \sum_{\substack{ j_0  = 2 \\ j_0 \neq j }}^{ k_2 +1} \, \sum_{i \in \PP_1(N_j) \cap  N_{j_0} } n_i , \\
& \al^{(1)}_j = \min(\al, d)  - 2 | \PP_1(N_1) \cap N_j | \bigg( \frac d2 - \be\bigg) - \eps.
\end{align*}

 \medskip
 
 \noi
 $\bullet$ {\bf Step 2:}
Next, we sum over $\{n_i\}$ for $i \in (\PP_1 \cap N_2) \setminus \PP_1(N_1)$. 
Given $i \in (\PP_1 \cap N_2) \setminus \PP_1(N_1)$, 
we have
$i \in  \PP_1(N_{j_0})$ for some $j_0 = 3, \dots, k_2 + 1$, 
and hence the frequency $n_i$ (with the $\pm$ sign)
appears in exactly  $\nf^{(1)}_2 $ and $\nf^{(1)}_{j_0}$ on the right-hand side of \eqref{Z35}.
Then, by applying Lemma~\ref{LEM:SUM}\,(ii), we have
\begin{align*}
S( \{n_i\}_{i  \notin \PP} ) \les \sum_{ \{ n_i \}  \in \PP_1 \setminus (\PP_1(N_1)  \cup N_2) } 
\bigg( \prod_{i \in \PP_1 \setminus (\PP_1(N_1)  \cup N_2)} \frac{1}{ \jb{n_i}^{2 \be}} \bigg)
 \prod_{j = 2}^{k_2 +1} \frac{1}{\jb{  \nf^{(2)}_j }^{\al^{(2)}_j}},
\end{align*}

\noi
provided that  $\be > \frac{d}{2} - \frac{\al}{2k}$, 
where $\nf^{(2)}_j$ and $\al^{(2)}_j$ are given by 
\begin{align*}
 \nf^{(2)}_2 & = \sum_{ i \in \PP^c \cap N_2 } n_i  - \sum_{j_0  = 3}^{ k_2 +1} \, \sum_{i \in \PP_1(N_2) \cap  N_{j_0} } n_i,  \\
 \al^{(2)}_2 & = \min(\al, d)  - \big(2  | \PP_1(N_1) \cap N_2 | + |( \PP_1 \cap N_2 ) \setminus \PP_1(N_1)| \big) \bigg( \frac d2 - \be\bigg) - \eps\\
& = \min(\al, d)  - \big( | \PP_1(N_1) \cap N_2 | + | \PP_1 \cap N_2 | \big) \bigg( \frac d2 - \be\bigg) - \eps, 
\end{align*}

\noi
when $j = 2$, 
and
\begin{align*}
& \nf^{(2)}_j = \sum_{ i \in \PP^c \cap N_j  } n_i + \sum_{i \in (\PP_1 \setminus \PP_1(N_1)) \cap N_j} n_i - \sum_{\substack{ j_0 =2 \\ j_0 \neq 2,j }}^{ k_2 +1} \, \sum_{i \in \PP_1(N_j) \cap  N_{j_0} } n_i \\
& \al^{(2)}_j =\min(\al, d)  - \big( 2 | \PP_1(N_1) \cap N_j | + | \PP_2(N_2) \cap N_j | \big) \bigg( \frac d2 - \be\bigg) - \eps, 
\end{align*}

\noi
when $3 \le j \le k_2 +1 $.

 \medskip
 
 \noi
 $\bullet$ {\bf Step 3:}
We iteratively sum over $\{n_i\}$ for $i \in N_j$ and $j = 3, \dots,  k_2 + 1$
and  obtain
\begin{align}
\S(\{n_i\}_{i  \notin \PP} ) \les \prod_{j = 2}^{k_2 +1} \frac{1}{\jb{\nf^{(k_2+1)}_j }^{\al^{(k_2+1)}_j}},
\label{Z38}
\end{align}

\noi
provided that  $\be > \frac{d}{2} - \frac{\al}{2k}$, 
where $\nf^{(k_2+1)}_j $ and $\al^{(k_2+1)}_j $, $j = 2, \dots, k_2+1$,  are given by 
\begin{align}
\begin{split}
& \nf^{(k_2+1)}_j = \sum_{ i \in \PP^c \cap N_j  } n_i,  \\
& \al^{(k_2+1)}_j = \min(\al, d)  - \big( | \PP_1(N_1) \cap N_j | + | \PP \cap N_j | \big) \bigg( \frac d2 - \be\bigg) - \eps. 
\end{split}
\label{Z38a}
\end{align}

\noi
Here, we used the following fact, which follows from 
\eqref{Z31}:
\begin{align*}
|\PP_1\cap N_j| + \sum_{\substack{j_0= 2\\j_0 \ne j}}^{k_2 + 1} |\PP_2(N_{j_0})\cap N_j| = |\PP \cap N_j|
\end{align*}

\noi
for each $j = 2, \dots, k_2 + 1$.

\medskip

By letting $m =  \sum_{ i \in \PP^c \cap N_1 } n_i$, 
it follows from $n = \sum_{i \notin \PP} n_i$
and \eqref{Z38a}
that 
\[ n - m = \sum_{j = 2}^{k_2 + 1} \nf^{(k_2+1)}_j.\]

\noi
Then, 
by combining \eqref{Z32} and \eqref{Z38}
and applying Lemma \ref{LEM:SUM}\,(i), we have
\begin{align}
\eqref{Z32} &  \les  \sum_{m \in \Z^d} \, 
\sum_{  \substack{ \{n_i\} \in \PP^c \cap N_1 \\ \sum_{ i \in \PP^c \cap N_1 } n_i = m } } \prod_{i \in \PP^c \cap N_1} \frac{1}{\jb{n_i}^{2\be}} \notag \\
& \hphantom{X} \times \sum_{  \substack{ \{ \nf^{(k_2+1)}_j \}_{2 \le j \le k_2 +1}  \\ \sum_{j=2}^{k_2 +1} \nf^{(k_2+1)}_j = n- m }}  \prod_{j=2}^{k_2 +1} \Bigg(
\frac{1}{  \jb{\nf^{(k_2+1)}_j} ^{2 \al^{(k_2+1)}_j}  }
\sum_{ \substack{  \{ n_i \} \in \PP^c \cap N_j  \\  \sum_{  i \in \PP^c \cap N_j }  n_i = 
\nf^{(k_2+1)}_j}} \prod_{ i \in \PP^c \cap N_j } \frac{1}{ \jb{ n_i} ^{2 \be}  } \Bigg)\notag \\
& \les \sum_{m \in \Z^d} \, \frac{1}{ \jb{ m }^{\be_1} } 
\sum_{  \substack{ \{ \nf^{(k_2+1)}_j \}_{2 \le j \le k_2 +1}  \\ \sum_{j=2}^{k_2 +1} 
\nf^{(k_2+1)}_j = n- m }}
  \prod_{j=2}^{k_2 +1} \frac{1}{  \jb{\nf^{(k_2+1)}_j} ^{2  \al^{(k_2+1)}_j + \be_j   } },  \label{Z41}
\end{align}

\noi
provided that $\be > \frac{d}{2} - \frac{\al}{2k}$, 
where $\be_j$ is defined by
\begin{align}
 \be_j =  d - 2 | \PP^c \cap N_j | \bigg( \frac{d}{2} - \be \bigg) - \eps.
 \label{Z41a}
\end{align}

\noi
Note that, for $2 \le j \le k_2 +1$, we have $|N_j| = k$ and thus 
\begin{align*}
2 \al_j^{(k_2+1)} + \be_j & = d + 2 \min(\al, d)- 2 \big( | \PP_1(N_1) \cap N_j |  + |N_j| \big) 
\bigg( \frac{d}{2} - \be \bigg)  - \eps \\
& > d + 2\min(\al, d) - 4 k \bigg( \frac{d}{2} - \be \bigg) > d,
\end{align*}

\noi
where the last inequality follows from the assumption \eqref{Zx1}
and $\be > \frac{k-1}{2k}d$.
Therefore,  by applying Lemma~\ref{LEM:SUM}\,(i)
with \eqref{Z41a} and $|N_1| = k_1$, 
we obtain
\begin{align*}
\eqref{Z41} & \les \sum_{m \in \Z^d} \, \frac{1}{ \jb{ m }^{\be_1}   \jb{n - m }^{d - \eps} } \\
& \les \jb{n}^{-d -  2k_1 ( \be - \frac{d}{2}   ) + \eps }, 
\end{align*}

\noi
provided that $\be > \frac{k_1-1}{2k_1}d$. 
Therefore, the desired regularity follows from Lemma \ref{LEM:reg}.
This concludes the proof of Proposition \ref{PROP:main}.
\end{proof}

\begin{remark}\label{REM:finite}\rm 
In the proof of Proposition \ref{PROP:sto1}\,(iii), 
we first treated the contribution from the highest homogeneous Wiener chaos
$\H_{k_1 + k k_2}$, 
and then treated the contribution from the lower order terms, 
which required more careful analysis.
Note that the analysis of the lower order terms yielded
an extra condition \eqref{Zx1}.
We also point out that
%
the regularity of the resulting stochastic term $Y  = \lim_{N \to \infty}Y_N$
is the same as that  coming from the contribution 
 from the highest homogeneous Wiener chaos.

\end{remark}

\subsection{Local well-posedness of the dispersion-generalized NLW}
\label{SUBSEC:3.3}

In this subsection, we prove local well-posedness
of  the renormalized
dispersion-generalized nonlinear wave equation~\eqref{NLW3}, 
using Proposition \ref{PROP:sto1}
and simple product estimates.
By writing~\eqref{NLW3} in the Duhamel formulation 
(with the zero initial data) and dropping the subscript $N$, 
we have 
\begin{align}
\begin{split}
v(t) = \Phi(v)(t) 
& = - 
\sum_{k_1 = 0}^{k-1} 
\sum_{k_2, k_3 = 0}^k\ind_{k = k_1 + k_2 + k_3} \frac{k!}{k_1!k_2!k_3!}\\
& \quad \times \int_0^t \frac{\sin((t - \tau) \jb{\nb}^\al)}{\jb{\nb}^\al} 
\Big(:\!Z^{k_1}\!:  \big(\I(:\!Z^k\!:)\big)^{k_2} v^{k_3}\Big)(\tau) d\tau.
\end{split}
\label{NLW4}
\end{align}

\noi
Recall that thanks to the second order expansion, 
we do not have the term with $k_1 = k$.
In view of Proposition \ref{PROP:sto1}, 
we first need to impose the condition 
\begin{align}
 \max \bigg(\frac d 2 - \frac{d}{2k}, \frac d2 - \frac \al {2k}\bigg)
< \be \le \frac d2
 \label{ZZ1}
\end{align}

\noi
to guarantee existence of the stochastic terms $:\!Z^{k_1}\!:\big(\I(:\!Z^k\!:)\big)^{k_2}$.

In the following, we prove local well-posedness of \eqref{NLW4}
by viewing 
\begin{align*}
\Xi = \big\{ :\!Z^{k_1}\!:  \big(\I(:\!Z^k\!:)\big)^{k_2}:  k_1 = 0, \dots, k-1, \ k_2 = 0, \dots, k\big\}
\end{align*}

\noi
as the collection of given deterministic distributions in
$C([0, T]; W^{s_{k_1, k_2}, \infty}(\T^d))$ with
(i)~$s_{k_1, k_2} < k_1(\be - \frac d2)$ 
for $k_1 \geq 1$ and (ii)~some small $s_{0, k_2} > 0$ when $k_1 = 0$.
Then, we set 
\begin{align*}
 \|\Xi\|_{\mathcal{Z}} = \sum_{k_1 = 0}^{k-1} \sum_{k_2 = 0}^k
 \big\| :\!Z^{k_1}\!:  \big(\I(:\!Z^k\!:)\big)^{k_2}\big\|_{C([0, 1]; W_x^{s_{k_1, k_2}, \infty})}.
\end{align*}

We now prove 
local well-posedness of the equation \eqref{NLW4}.

\begin{theorem}\label{THM:LWP}
Fix an integer $k \ge 2$. Let $\al > 0$ and  $\be \le \frac d2$, 
satisfying \eqref{ZZ1}
and 
\begin{align}
\al > (k-1) \bigg(\frac{d}{2} -\frac{k-1}{k}\be\bigg).
\label{C1}
\end{align}

\noi
Furthermore, when $k \ge 3$, we assume that
\begin{align}
\al > (k-1) \frac d2 - \frac{k^2 - 3k + 3}{k-1} \be, 
\label{C2}
\end{align}

%
%

\noi
which is a stronger assumption than \eqref{C1}.
Then, there exists a unique local-in-time solution~$v$ to 
\eqref{NLW4}  in the class
 $C([0, T]; H^\s(\T^d))$ for some $\s > 0$
and $T = T(\|\Xi\|_{\mathcal{Z}}) > 0$.
Moreover,  the map$:\Xi \in \mathcal{Z}\mapsto v \in C([0, T]; H^\s(\T^d))$  
is continuous.

\end{theorem}

In view of the second order expansion \eqref{exp2}
and Proposition \ref{PROP:sto1}, 
Theorem \ref{THM:LWP} yields
 almost sure local well-posedness
of  the renormalized dispersion-generalized NLW \eqref{NLW1}
with the random initial data $(u_0^\o, u_1^\o)$
in \eqref{series} 
in the following sense.
Given $N \in \N$, let $u_N$ be 
the solution to the renormalized dispersion-generalized NLW
with the truncated random initial data:
\begin{align*}
\begin{cases}
\dt^2 u_N + (1-\Dl)^\al u_N + H_k(u_N; \s_N) = 0\\
(u_N, \dt u_N )|_{t = 0} = (\P_N u_0^\o, \P_Nu_1^\o)
\end{cases}
\end{align*}

\noi
We have $u_N = Z_N -  \I(:\!Z_N^k\!:) + v_N$, 
where $v_N$ solves \eqref{NLW3}.
Then, $u_N$ converges 
to some non-trivial process $u = Z -  \I(:\!Z^k\!:)  +  v$ in 
$C([0, T_\o]; H^\s(\T^d))$
almost surely, where $T_\o$ is an almost surely positive random 
time, 
$Z$ is the random linear solution in \eqref{NLW2}, 
and $v$ satisfies the nonlinear equation~\eqref{NLW4}.
Here, 
$ \I(:\hspace{-0.9mm}Z^k\hspace{-0.9mm}:) $ is the limit of  
$\I(:\hspace{-0.9mm}Z_N^k\hspace{-0.9mm}:)$ whose existence is guaranteed in
Proposition \ref{PROP:sto1}.

When $d = 2$ and $\al = \be = 1$, 
Thomann and the first author \cite{OTh2} proved Theorem \ref{THM:LWP}, 
(corresponding to the Gibbs measure problem for the standard nonlinear wave equations
on~$\T^2$).
See also \cite{GKO, OOcomp, OOTz}.
When $d = 2$ and $\be = 1$, Theorem \ref{THM:LWP} extends the result in~\cite{OTh2}
to (i)~$\al > \al_1(k) = \frac{k-1}{k} = \frac 12$ when $k = 2$
and (ii)~$\al > \al_2(k) =  \frac{k-2}{k-1}$ when $k \ge 3$.
Note that, for $j = 1, 2$,  we have $\al_j(k) \to 1$ as $k\to \infty$.
See also \cite{STzW, FPT} for $\al < 1$ and $d = 2$. 
In the following, we prove Theorem~\ref{THM:LWP}
by using the product estimates (Lemma~\ref{LEM:bilin})
and Sobolev's inequality.
It is possible to improve Theorem~\ref{THM:LWP}
by using the Strichartz estimates
but we do not pursue this issue here,  since our main focus in this section 
is to present a simple application of Proposition~\ref{PROP:main}
(and Proposition~\ref{PROP:sto1}).
See \cite{FPT}
for a local well-posedness argument with the Strichartz estimates
when $ d= 2$.

\smallskip

Before proceeding to the proof of Theorem \ref{THM:LWP}, 
we first recall the following product estimates.
See \cite{GKO} for the proof.
Note that
while  Lemma \ref{LEM:bilin} (ii) 
was shown only for 
$\frac1p+\frac1q= \frac1r + \frac{s}d $
in \cite{GKO}, 
the general case
$\frac1p+\frac1q\leq \frac1r + \frac{s}d $
follows from the inclusion $L^{r_1}(\T^d)\subset L^{r_2}(\T^d)$
for $r_1 \geq r_2$.

\begin{lemma}\label{LEM:bilin}
 Let $0\le s \le 1$.\rule[-2mm]{0mm}{0mm}
\\
\textup{(i)} Suppose that 
 $1<p_j,q_j,r < \infty$, $\frac1{p_j} + \frac1{q_j}= \frac1r$, $j = 1, 2$. 
 Then, we have  
\begin{equation*}  
\| \jb{\nb}^s (fg) \|_{L^r(\T^d)} 
\les \Big( \| f \|_{L^{p_1}(\T^d)} 
\| \jb{\nb}^s g \|_{L^{q_1}(\T^d)} + \| \jb{\nb}^s f \|_{L^{p_2}(\T^d)} 
\|  g \|_{L^{q_2}(\T^d)}\Big).
\end{equation*}

\smallskip

\noi
\textup{(ii)} 
Suppose that  
 $1<p,q,r < \infty$ satisfy the scaling condition:
$\frac1p+\frac1q\leq \frac1r + \frac{s}d $.
Then, we have
\begin{align*}
\big\| \jb{\nb}^{-s} (fg) \big\|_{L^r(\T^d)} \les \big\| \jb{\nb}^{-s} f \big\|_{L^p(\T^d) } 
\big\| \jb{\nb}^s g \big\|_{L^q(\T^d)}.  
\end{align*}

\end{lemma}

\medskip

We now present the proof of Theorem \ref{THM:LWP}.
In the following, 
we use the short-hand notation
$C_T B_x$ for $(C([0, T]; B(\T^d))$.

\begin{proof}[Proof of Theorem  \ref{THM:LWP}]
We show that $\Phi$ in \eqref{NLW4}
is a contraction in a ball   $ B\subset C([0, T]; H^\s(\T^d))$ of radius 1
 centered at the origin.

 Let us first treat the stochastic terms.
The worst term on the right-hand side of \eqref{NLW4}
 is
$ :\!Z^{k-1}\!:   \I(:\!Z^k\!:)$
with regularity\footnote{In the following, we
only discuss spatial regularities.
We also use $a-$ to denote $a-\eps$ for arbitrarily small $\eps > 0$.} 
$ (k-1) (\be - \frac d2) -$.
In view of the gain of $\al$-regularity from the Duhamel integral operator, 
this implies that $v$  
 is expected to have regularity 
 \begin{align}
 \s < \al + (k-1) \bigg(\be - \frac d2\bigg).
 \label{ZZ2}
 \end{align}

\noi
Note that, thanks to   the hypothesis \eqref{C1}, 
we can choose small $\s > 0$, satisfying \eqref{ZZ2}.
Under the condition  \eqref{ZZ2}, 
 we have 
 \begin{align}
\begin{split}
\bigg\| \int_0^t  & \frac{\sin((t - \tau) \jb{\nb}^\al)}{\jb{\nb}^\al} 
\Big(:\!Z^{k_1}\!:  \big(\I(:\!Z^k\!:)\big)^{k-k_1}\Big)(\tau) d\tau\bigg\|_{C_TH^\s_x}\\
&  \leq T
\| :\!Z^{k_1}\!:  \big(\I(:\!Z^k\!:)\big)^{k-k_1}\|_{C_TH^{\s-\al}_x}
 < \infty
\end{split}
\label{ZZ2a}
\end{align}

\noi
for any $k_1 = 0, 1, \dots, k - 1$.

On the other hand, 
by Sobolev's inequality, 
H\"older's inequality, and Sobolev's inequality once again, 
we estimate 
 the contribution from $k_1 = k_2 = 0$ and $k_3 = k$ as
 \begin{align}
\begin{split}
\bigg\|\int_0^t \frac{\sin((t - \tau) \jb{\nb}^\al)}{\jb{\nb}^\al} 
 v^{k}(\tau) d\tau\bigg\|_{C_TH^\s_x}
&  \le T\| v^k\|_{C_TH^{\s-\al}_x}
  \les T\| v\|_{C_TL^{kp}_x}^k\\
&\les T \| v\|_{C_TH^{\s}_x}^k.
\end{split}
\label{ZZ3}
\end{align}

\noi
The second inequality in \eqref{ZZ3} holds under the condition
\begin{align}
\frac{\al - \s}{d} \ge \frac{1}{p} - \frac{1}{2}
\label{ZZ4}
\end{align}
 
\noi
with some $ p \le 2$, 
while the third inequality holds under the condition
\begin{align}
\frac{\s}{d} \ge  \frac 12-\frac1{kp}. 
\label{ZZ5}
\end{align}

Next, we consider 
 a general term $:\!Z^{k_1}\!:  \big(\I(:\!Z^k\!:)\big)^{k_2} v^{k_3}$
with $1 \le k_3 \le k-1$.
Noting that the regularity of 
$:\!Z^{k_1}\!:  \big(\I(:\!Z^k\!:)\big)^{k_2}$
depends only on $k_1$ (and gets worse as $k_1$ gets larger), 
it suffices to consider 
$:\!Z^{k_1}\!:   v^{k - k_1}$
with $1 \le k_1 \le k-1$, 
 since when $k_2 > 0$, 
the value of $k_1$ or $k_3$ is smaller, which makes it easier to handle the term.

By 
Sobolev's inequality, 
Lemma \ref{LEM:bilin}\,(ii), 
the fractional Leibniz rule (Lemma \ref{LEM:bilin}\,(i)), 
H\"older's inequality, 
and Sobolev's inequality, we have 
 \begin{align}
\begin{split}
\bigg\| \int_0^t  & \frac{\sin(t - \tau) \jb{\nb}^\al)}{\jb{\nb}^\al} 
\Big(:\!Z^{k_1}\!:  v^{k - k_1}\Big)(\tau) d\tau\bigg\|_{C_TH^\s_x}
 \leq 
\| :\!Z^{k_1}\!:  v^{k-k_1}\|_{C_TH^{\s-\al}_x}\\
& \les
T\| :\!Z^{k_1}\!:  v^{k-k_1}\|_{C_TW^{k_1(\be - \frac d2)-, q}_x}\\
&  \les
T\| :\!Z^{k_1}\!: \|_{C_TW^{k_1(\be - \frac d2)-, \infty}_x}
\|  v^{k-k_1}\|_{C_TW^{k_1( \frac d2 - \be)+, q_0}_x}\\
&  \les
T\| :\!Z^{k_1}\!: \|_{C_TW^{k_1(\be - \frac d2)-, \infty}_x}
\| v\|_{C_T W^{k_1( \frac d2 - \be)+, r_0}_x}
\| v\|_{C_TL^{r}_x}^{k-k_1 - 1}\\
&  \les
T\| :\!Z^{k_1}\!: \|_{C_TW^{k_1(\be - \frac d2)-, \infty}_x}
\| v\|_{C_T H^\s_x}
\| v\|_{C_TL^{r}_x}^{k-k_1 - 1}\\
&  \les
T\| :\!Z^{k_1}\!: \|_{C_TW^{k_1(\be - \frac d2)-, \infty}_x}
\| v\|_{C_T H^\s_x}^{k  - k_1}, 
\end{split}
\label{ZZ6}
\end{align}

\noi
where we omit the step with the fractional Leibniz rule 
when $k_1 = k-1$.
We now list the conditions needed for each step in \eqref{ZZ6}.
At the second inequality in \eqref{ZZ6}, 
we used Sobolev's inequality with 
\begin{align}
\al - \s > k_1 \bigg(\frac d2 - \be \bigg)
\qquad \text{and}\qquad 
\al - \s + k_1 \bigg(\be - \frac d2\bigg)>\frac dq - \frac d2
\label{ZZ7}
\end{align}

\noi
(note that the first condition in \eqref{ZZ7} is subsumed by 
\eqref{ZZ2}).
At the third inequality in~\eqref{ZZ6}, 
we applied Lemma \ref{LEM:bilin}\,(ii) with 
\begin{align}
k_1 \bigg(\frac d2 -\be \bigg) + \frac dq
\ge \frac d {q_0}.
\label{ZZ8}
\end{align}

\noi
The fourth inequality in \eqref{ZZ6} follows from 
the fractional Leibniz rule (Lemma \ref{LEM:bilin}\,(i))
with 
\begin{align}
\frac 1{q_0} = \frac 1{r_0} + \frac{k-k_1 - 1}{r}.
\label{ZZ9}
\end{align}

\noi
The fifth inequality in \eqref{ZZ6} follows from 
Sobolev's inequality with
\begin{align}
\s - k_1\bigg( \frac d2 - \be\bigg) > \max\bigg(\frac d2 - \frac d{r_0}, 0\bigg).
\label{ZZ9a}
\end{align}

\noi
At the last inequality in \eqref{ZZ6}, 
we used Sobolev's inequality with 
\begin{align}
\frac{\s}{d} \ge \frac 12 -\frac 1r.
\label{ZZ10}
\end{align}

\noi
We point out that, by taking $\s$ sufficiently close to $\al >0$  in  \eqref{ZZ2},
 the second condition in~\eqref{ZZ9a}:
$\s - k_1\big( \frac d2 - \be\big) >  0$
is satisfied; see also \eqref{Q3} and \eqref{Q3c} below.

Assuming all the conditions, 
it follows from 
\eqref{ZZ2a}, \eqref{ZZ3}, and \eqref{ZZ6}
(which can be applied to 
a general product $:\!Z^{k_1}\!:  \big(\I(:\!Z^k\!:)\big)^{k_2} v^{k_3}$
with $1 \le k_3 \le k-1$)
that the map $\Phi$ in \eqref{NLW4}
is bounded on 
the ball 
 $ B\subset C([0, T]; H^\s(\T^d))$ by choosing $T = T(\|\Xi\|_{\mathcal{Z}}) >0$
 sufficiently small.
A similar computation yields a difference estimate
on $\Phi(v_1) - \Phi(v_2)$,
from which we conclude that 
that the map $\Phi$ in \eqref{NLW4}
is a contraction on 
the ball 
 $ B\subset C([0, T]; H^\s(\T^d))$ by choosing sufficiently small $T = T(\|\Xi\|_{\mathcal{Z}}) >0$.

We conclude the proof by summarizing the conditions
and show that there exists $\s > 0$ satisfying all the conditions.
From \eqref{ZZ4} and \eqref{ZZ5}, we have
\begin{align}
\al \ge (k-1) \bigg(\frac{d}{2} - \s\bigg).
\label{Q1}
\end{align}

\noi
From \eqref{ZZ7}, \eqref{ZZ8}, \eqref{ZZ9}, \eqref{ZZ9a}, and \eqref{ZZ10} (with $k_1 = 1$)
\begin{align*}
\al - \s & \ge 
\frac dq - \frac d2 + k_1 \bigg(\frac d2 - \be \bigg)  \ge \frac d{q_0} - \frac d2\notag \\
& = \frac d{r_0} + \frac{(k - k_1 - 1)d}{r} - \frac d2 \notag \\
& > -\s + k_1 \bigg(\frac d2 - \be\bigg) + 
(k-k_1 - 1) \bigg(\frac d2 - \s\bigg)
\end{align*}

\noi
and thus we have
\begin{align}
\al > k_1(\s-\be) + (k-1)\bigg(\frac d2 - \s\bigg)
\label{Q2}
\end{align}

\noi
for any $k_1 = 1, \dots, k - 1$.

From \eqref{ZZ2} and \eqref{Q1}, we obtain
\begin{align}
\frac{d}{2} - \frac \al{k-1}  \le  \s < \al + (k-1) \bigg(\be - \frac d2\bigg).
\label{Q3} 
\end{align}

\noi
The hypothesis \eqref{C1} guarantees existence of $\s > 0$ 
 satisfying \eqref{Q3}.
Note that the condition~\eqref{C1} guarantees that the right-hand side of \eqref{Q3}
is positive.

Next, we verify existence of $\s > 0$, 
satisfying \eqref{Q2} for $k_1 = 1, \dots, k-1$, 
under \eqref{ZZ2}.
For this purpose, we only need to consider 
the following two extreme  cases:
$k_1 = k-1$ and $k_1 = 1$.
From~\eqref{ZZ2} and~\eqref{Q2} with $k_1 = k-1$, we obtain
\begin{align*}
\al > (k-1) \bigg(\frac{d}{2} - \be\bigg), 
\end{align*}

\noi
which is guaranteed by the hypothesis \eqref{C1}.
Next, we consider 
the consequence
of~\eqref{ZZ2} and \eqref{Q2} with $k_1 = 1$.
When $k = 2$, 
it follows from \eqref{ZZ2} and \eqref{Q2} with $k_1 = 1$ that 
\begin{align*}
\al >  \frac{d}{2} - \be = (k-1) \bigg(\frac{d}{2} - \be\bigg), 
\end{align*}

\noi
which is once again guaranteed by the hypothesis \eqref{C1}.

Finally, let $k \ge 3$.
Then, 
 from \eqref{ZZ2} and \eqref{Q2} with $k_1 = 1$, we have
 \begin{align}
\frac{k-1}{k-2}\frac{d}{2} - \frac{\be}{k-2} - \frac \al{k-2}  \le  \s < \al + (k-1) \bigg(\be - \frac d2\bigg), 
\label{Q3c} 
\end{align}

\noi
The hypothesis \eqref{C2} guarantees existence of $\s > 0$ 
 satisfying \eqref{Q3c}.
 This concludes the proof of Theorem \ref{THM:LWP}.
\end{proof}

\begin{ackno}\rm

T.O.~and Y.Z.~were supported by the European Research Council (grant no.~864138 ``SingStochDispDyn''). 
The authors would like to thank the anonymous referee
for the helpful comments.

\end{ackno}

\end{document}